\documentclass[final, 12pt, leqno]{amsart}
\usepackage[utf8x]{inputenc}
\usepackage{amsmath, amsthm, amsfonts, amssymb, enumerate, wrapfig,
tikz,  color, rotating, colortbl,   datetime,  float,  stmaryrd,
hyperref, bookmark, cleveref, ifthen, tikz-3dplot, setspace, mathtools, xcolor}
\usetikzlibrary{decorations.markings, decorations.pathreplacing, decorations.pathmorphing}
\usepackage{amscd}
\title{Conical \texorpdfstring{SL(3)}{SL(3)} foams}
 \author{Mikhail Khovanov} 
 \address{Department of Mathematics, Columbia University, New York, NY 10027, USA}
 \email{\href{mailto:khovanov@math.columbia.edu}{khovanov@math.columbia.edu}}
 \author{Louis-Hadrien Robert}
 \address{Université du Luxembourg, RMATH, 5, avenue de la Fonte, L-4365 Esch-sur-Alzette, Luxemburg}
 \email{\href{mailto:louis-hadrien.robert@un.lu}{louis-hadrien.robert@uni.lu}}
 \hypersetup{
    pdftoolbar=true,        
    pdfmenubar=true,        
    pdffitwindow=false,     
    pdfstartview={FitH},    
    pdftitle={\shorttitle},    
    pdfkeywords={}, 
    pdfnewwindow=true,      
    colorlinks=true,       
    linkcolor=red,          
    citecolor=teal,        
    filecolor=magenta,      
    urlcolor=violet,          
    linkbordercolor=red,
    citebordercolor=teal,
    urlbordercolor=violet, 
    linktocpage=true}
\usetikzlibrary{arrows, decorations.markings, decorations, patterns, positioning, decorations.pathreplacing, decorations.pathmorphing}
\tikzset{->-/.style={decoration={markings, mark=at position .5 with {\arrow{>}}},postaction={decorate}}}
\tikzset{-<-/.style={decoration={markings, mark=at position .5 with {\arrow{<}}},postaction={decorate}}}
\let\oldtocsubsection\tocsubsection
\renewcommand\tocsubsection[3]{\hspace{0.5cm}\oldtocsubsection{#1}{#2}{#3}}
\let\oldtocsubsubsection\tocsubsubsection
\renewcommand\tocsubsubsection[3]{\hspace{1cm}\oldtocsubsubsection{#1}{#2}{#3}}

\let\emptyset\varnothing

\newcounter{res}[section]
\numberwithin{res}{section}

\newtheorem{lem}[res]{Lemma}

\newtheorem{prop}[res]{Proposition}

\theoremstyle{definition}

\newtheorem{dfn}[res]{Definition}
\newtheorem{rmk}[res]{Remark}
\newtheorem{exa}[res]{Example}

\newcommand{\imagesfolder}{Images}
\def\co{\colon\thinspace}

\newcommand{\NB}[1]{\ensuremath{\vcenter{\hbox{#1}}}}

\newcommand{\FF}{\ensuremath{\mathbb{F}}}
\newcommand{\ZZ}{\ensuremath{\mathbb{Z}}}
\newcommand{\CC}{\ensuremath{\mathbb{C}}}

\newcommand{\RR}{\ensuremath{\mathbb{R}}}

\renewcommand{\SS}{\ensuremath{\mathbb{S}}}
  
\newcommand{\Hom}{\mathop{\mathrm{Hom}}}

\newcommand{\rk}{\mathrm{rk}}

\newcommand\kup[1]{\left\langle #1 \right\rangle}

\newcommand{\vect}{\ensuremath{\mathsf{vect}}}
\newcommand{\md}{\ensuremath{\mathsf{mod}}}

\newcommand{\Foam}{\ensuremath{\mathsf{Foam}}}

\newcommand{\col}[2][{}]{\ensuremath{\mathrm{adm}_{#1}(#2)}} 
\newcommand{\R}{\ensuremath{R}} 
\newcommand{\Tait}{\mathrm{Tait}}
\newcommand{\cone}{\mathrm{Cone}}
\newcommand{\Z}{\mathbb{Z}}
\newcommand{\Gd}{\ensuremath{G_d}}
\newcommand{\lra}{\longrightarrow}
\newcommand{\kk}{\mathbf{k}}

\newcommand{\leftsquigarrow}{\ensuremath{\rotatebox[origin=c]{180}{\NB{$\rightsquigarrow$}}}}

\newcommand{\cfthetacol}[4][0.5]{\NB{\tikz[font=\tiny,scale=#1]{
\begin{scope}
\draw (0,1) arc (90:270:1cm) node[pos= 0.5, left] {$#2$};
\draw (0,1) arc (90:-90:1cm) node[pos= 0.5, right] {$#4$};
\draw (0,1) -- (0, -1) node[pos=0.5, right] {$#3$};
\end{scope}}}}

\date{November 22, 2020}

\begin{document}
\begin{abstract}
In the unoriented SL(3) foam theory, singular vertices are generic singularities of two-dimensional complexes. Singular vertices have neighbourhoods homeomorphic to cones over the one-skeleton of the tetrahedron, viewed as a trivalent  graph  on the two-sphere. In this paper we consider foams with singular vertices with neighbourhoods homeomorphic to cones over more general planar trivalent graphs. These graphs are subject to suitable conditions on their Kempe equivalence Tait coloring classes and include the dodecahedron graph. In this modification of the original homology theory it is straightforward to show that modules associated to the dodecahedron graph are free of rank 60, which is still an open problem for the original unoriented SL(3) foam theory. 
\end{abstract}
\maketitle
\tableofcontents
\section{Introduction}

\vspace{0.1in} 

Kronheimer and Mrowka \cite{KM1} defined a functor $J^\sharp$ from the category of knotted trivalent graphs (KTG) in $\RR^3$ and foams in $\RR^3\times [0,1]$ to the category of $\kk$-vector spaces, where $\kk=\FF_2$ is the two-element field. They proved that if $\Gamma$ is a bridgeless KTG embedded in the plane, then $J^\sharp(\Gamma)$ is nontrivial. Kronheimer and Mrowka proved that for such a KTG the inequality  $\dim(J^\sharp(\Gamma))\geq \left|\mathrm{Tait} (\Gamma)\right|$ holds and conjectured that it is an equality. Here $\mathrm{Tait}(\Gamma)$ is the set of Tait colorings  of $\Gamma$, that is, 3-colorings of edges of $\Gamma$ such that at each vertex the three colors are distinct.  The number of Tait colorings of $\Gamma$ is 
denoted $\left|\mathrm{Tait} (\Gamma)\right|$. 

This conjecture would imply the four color theorem (4CT) \cite{ApHa}. In the same paper, for plane graphs $\Gamma\subset \RR^2$, Kronheimer and Mrowka predicted the existence of a combinatorial analogue of $J^\sharp(\Gamma)$, denoted $J^\flat(\Gamma)$ in~\cite{KM1} and $\kup{\Gamma}$ here. The existence of $J^\flat(\Gamma)\cong \kup{\Gamma}$ was proved by the authors in~\cite{KR1}.

Homology groups $\kup{\Gamma}$, over all plane trivalent graphs $\Gamma$, extend to a functor from the category of foams with boundary to the category of $\Z$-graded modules over a suitable commutative ring. We denote this functor by $\kup{\bullet}$. 

In fact,  several versions of the functor $\kup{\bullet}$ are constructed in~\cite{KR1}. One of them takes values in the category of $\Z$-graded $\kk$-vector spaces, another one---in the category of graded free  $\kk[E]$-modules, where $\deg{E}=6$. Here $\kk$ is a characteristic two  field, which most of the time can be set to $\FF_2$. The most general homology of this type defined in~\cite{KR1} takes values in the category of graded modules over the ring $R=\kk[E_1,E_2,E_3]$ of symmetric polynomials in $X_1,X_2,X_3$, where
$E_1,E_2,E_3$ are the elementary symmetric  functions in these three  variables and $\deg(X_i)=2$, $i=1,2,3$. In this paper we use the following three theories.  
\begin{itemize}
    \item Homology theory $\kup{\bullet}$ with state spaces -- graded $R$-modules, with $R=\kk[E_1,E_2,E_3]$ as  above. This theory is based on 
    the original foam evaluation in~\cite{KR1}. 
    \item Homology theory $\kup{\bullet}_{0}$ with state spaces -- graded $\kk$-vector spaces. It is defined by evaluating closed foams $F$ to elements $\kup{F}_{0}$ of $\FF_2=\{0,1\}\subset \kk$ via the composition of the original evaluation  $\kup{F}\in R$ and the homomorphism $R\lra \kk\cong R/(E_1,E_2,E_3)$. This evaluation just picks the constant term in the evaluation $\kup{F}$ of a closed foam, viewed as a symmetric polynomial in $X_1,X_2,X_3$. 
    \item Homology theory $\kup{\bullet}_E$ with state spaces -- graded $\kk[E]$-modules.
This theory is defined via the evaluation  taking closed foams $F$ to elements $\kup{F}_{E}$ of $\kk[E]$ via the composition of the original evaluation  $\kup{F}\in R$ and the homomorphism $R\lra \kk[E]$ taking $E_1,E_2$ to $0$ and $E_3$ to $E$.
\end{itemize}

There are functorial in $\Gamma$ homomorphisms (base change homomorphisms) 
\begin{equation}
     \kup{\Gamma}\otimes_R \kk[E]\lra \kup{\Gamma}_E, \ \ \kup{\Gamma}\otimes_{R}\kk \lra \kup{\Gamma}_0  
     \, . 
\end{equation}

\begin{figure}
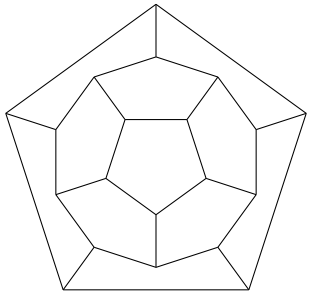

    \centering
    \NB{\tikz[]{\begin{scope}[yscale = {0.7}, xscale={0.7},decoration={markings, mark=at
  position 0.5 with {\arrow{>}}},postaction={decorate}]
\draw (18:3) -- (90:3);
\draw (90:3) -- (162:3);
\draw (162:3) -- ( 234:3);
\draw (234:3) -- (306:3);
\draw (306:3) -- (18:3);
\draw (18:3) -- (18:2);
\draw (90:3) -- (90:2);
\draw (162:3) -- (162:2);
\draw (234:3) -- (234:2);
\draw (306:3) -- (306:2);
\draw (-18:1) -- (54:1);
\draw (54:1) -- (126:1);
\draw (126:1) -- ( 198:1);
\draw (198:1) -- (270:1);
\draw (270:1) -- (-18:1);
\draw (-18:1) -- (-18:2);
\draw (54:1) -- (54:2);
\draw (126:1) -- (126:2);
\draw ( 198:1) -- ( 198:2);
\draw (270:1) -- (270:2);
\draw (-18:2) -- (18:2);
\draw (18:2) -- (54:2);
\draw (54:2) -- (90:2);
\draw (90:2) -- (126:2);
\draw (126:2) -- (162:2);
\draw (162:2) -- (198:2);
\draw (198:2) -- (234:2);
\draw (234:2) -- (270:2);
\draw (270:2) -- (306:2);
\draw (306:2) -- (-18:2);
\end{scope}}}
    \caption{The dodecahedral graph $\Gd$.}
    \label{fig:dodecahedron}
\end{figure}

The smallest graph for which none of $\kup{\bullet},\kup{\bullet}_E, \kup{\bullet}_{0}$, or $J^\sharp$ is fully understood is the dodecahedral graph depicted in Figure~\ref{fig:dodecahedron} and denoted $\Gd$. 

\vspace{0.1in} 

From~\cite[Proposition~4.18]{KR1} and the behaviour of the  theory under the graded localization $R\lra R[\mathcal{D}^{-1}]$, where $\mathcal{D}=E_3-E_1E_2$ is the square root of the discriminant, we know that $\kup{\Gd}_E$ is a free graded $\kk[E]$-module of rank $60$ (the number of Tait colorings of $G_d$), but we do not know its graded rank, which is an element of  $\ZZ_+[q,q^{-1}]$. The naive guess is that the graded rank of $\kup{\Gd}_E$ is $10[3][2]=10(q^{-2}+1+q^2)(q+q^{-1})$. 

Potentially, the rank or dimension of the state space of a graph $\Gamma$ can drop upon changing the ground ring due to the lack of commutativity between the base change operation on commutative rings and taking the quotient of the bilinear form on  a free module by its quotient. The simplest example is given by the bilinear form $(E)$ on rank one free $\kk[E]$-module $V$. The kernel  is zero and the quotient by the kernel is $V$. Under the base change $\kk[E]\lra \kk$ taking $E$ to $0$ bilinear form becomes $(0)$, and the quotient by the kernel is the trivial $\kk$-vector space. The latter is different from $V\otimes_{\kk[E]}\kk\cong \kk$, showing a drop in dimension.  

David Boozer~\cite{boozer2019computer} wrote a program to compute approximations to the modules $\kup{\Gd}$. His results show that $\dim \kup{\Gd}_{0}\ge 58$, which is, however, strictly smaller than 60, the number of Tait colorings of $\Gd$. Boozer's program looks at several thousand foams bounded by $\Gd$ and computes the rank of the bilinear form used to define $\kup{\bullet}_{0}$ restricted to the subspace spanned by these foams. The rank is $58$, not $60$, indicating  that the subspace spanned by these foams in $\kup{\Gd}_{0}$ has dimension $58$. An even stronger guess based on this data would be that $\kup{\Gd}_{0}$ has dimension $58$. Boozer's work also suggests~\cite{boozer2019computer} that the graded rank of $\kup{\Gd}_E$ is $9q^{-3} + 20q^{-1} + 20 q^1  + 11q^{3}$, rather than the  naive guess given by formula (\ref{eq_10_3}) below, 
and that the natural map $\kup{\Gd}_E\otimes_{\kk[E]}\kk\lra \kup{\Gd}_0$ has a nontrivial kernel. 

 It follows from~\cite{KM1,KR1} that $\kup{\Gamma}_0$ is naturally a subquotient of $J^{\sharp}(\Gamma)$ and $
 \dim(J^\sharp(\Gamma))\ge \dim(\kup{\Gamma}_0)$. 
 
Kronheimer and Mrowka~\cite{KM2, KM3} proved that  $\dim J^{\sharp}(\Gamma)\geq
 \left| \mathrm{Tait}(\Gamma) \right|$ for any graph $\Gamma$, in particular,
 $\dim J^{\sharp}(\Gd) \ge 60=\left| \mathrm{Tait}(\Gd) \right|$. 
 
This may suggest that $J^\sharp(\Gd) \ncong  \kup{\Gd}_{0}$. At the  same time, one
can prove that $\rk \kup{\Gamma}_E =
\left|\mathrm{Tait}(\Gamma)\right|$ for any $\Gamma$ and, consequently, $\kup{\Gamma}_E$ is a free graded  $\kk[E]$-module of rank $\left|\mathrm{Tait}(\Gamma)\right|$. The graded rank of $\kup{\Gamma}_E$, an element of $\Z_+[q,q^{-1}]$, seem complicated to determine or understand, though, already for $G_d$.

In this paper we enhance the
domain category of $\kup{\bullet}$ by extending the foams under
consideration to \emph{conical foams}. Category of foams $\Foam$ gets enlarged to the category $\Foam^c$ of conical foams that comes with 
evaluation functors $\kup{\bullet}^c$,  $\kup{\bullet}^c_0$ and $\kup{\bullet}^c_{E}$ into categories of graded modules over rings $R$, $\kk$  and $\kk[E]$, correspondingly.  The superscript $c$ stands for  \emph{conical}. 
One advantage of this construction is that for the dodecahedral graph the resulting homology groups have the desired size: 


\begin{prop}\label{prop:main}
    The graded dimension or rank of vector spaces or modules $\kup{\Gd}^c_0$, $\kup{\Gd}^c$,  and $\kup{\Gd}^c_E$ is   
    \begin{equation}\label{eq_10_3}
        10[2][3]=10(q+ q^{-1})(q^2 + 1 + q^{-2}).
    \end{equation}
\end{prop}

Foams of a new type introduced in this paper have singularities which are cones over planar trivalent graphs endowed with an equivalence class of Tait colorings of this graph satisfying a technical condition. Our construction does not answer the original question about the difference in dimensions for 
$J^\sharp(\Gd)$ and $\kup{\Gd}_{0}$. Rather, it enlarges the state spaces $\kup{\Gamma}_0$ to $\kup{\Gamma}^c_0$ to achieve the desired dimension $60$ for the dodecahedral graph.


We do not know whether $\kup{\Gamma}^c$ is larger than $\kup{\Gamma}$ for some graph  $\Gamma$. We also don't know the answer to the same question for the $\kup{\Gamma}_0^c$ versus
$\kup{\Gamma}_0$ or $\kup{\Gamma}^c_E$ versus $\kup{\Gamma}_E$. The first potential example to further investigate is $\Gd$. 


  In section~\ref{sec:conical-foams}  the foam evaluation from \cite{KR1} is extended to conical foams. We construct functors $\kup{\bullet}^c$,  $\kup{\bullet}^c_0$ and $\kup{\bullet}^c_{E}$ from the category of conical foams $\Foam^c$ to categories of graded $R$-modules, $\kk$-vector spaces and $\kk[E]$-modules. 
  In section~\ref{sec:dodecahedron} we compute the state spaces of the dodecahedron and prove Proposition~\ref{prop:main}. In section~\ref{sec:kempe} we discuss the technical condition which appear in the definition of conical foams.

\vspace{0.1in} 

{\bf Acknowledgments:} M.K. was partially supported by the NSF grant DMS-1807425 while working on this paper.



\section{Conical foams}
\label{sec:conical-foams}
\subsection{Tait colorings and Kempe moves}
\label{sec:tait-kempe}
A \emph{web} is a finite trivalent graph PL-embedded in $\RR^2$. It may have multiple edges between a  pair of  vertices and  loops. Furthermore, loops without vertices are allowed  as well. A \emph{Tait coloring} of a web $\Gamma= (V(\Gamma), E(\Gamma))$ is a map $c\co E(\Gamma) \to I_3$ from  the  set of  edges to the 3-element set $I_3=\{1,2,3\}$  such that at each vertex the three adjacent edges are mapped to distinct elements in $I_3$. A vertexless loop counts as an edge, an element of $E(\Gamma)$, and under $c$ is also mapped to an element of $I_3$. The set of Tait colorings of a given web $\Gamma$ is denoted by $\Tait(\Gamma)$. There is a natural action of $S_3$ on $\Tait(\Gamma)$ by permutations of colors. 

If $c$ is a coloring of $\Gamma$ and $i$ and $j$ are two distinct elements of $\{1,2,3\}$, denote by $\Gamma_{ij}(c)$ the graph $(V(\Gamma), E_{ij}(\Gamma,c))$, where
\[
E_{ij}(\Gamma,c) = \left\{ e \in E(\Gamma) \  | \ c(e) \in \{i, j\} \right\}.
\]
The graph $\Gamma_{ij}(c)$ is a bivalent graph and therefore a disjoint union of cycles of even lengths, possibly including length zero (when a circle of $\Gamma$ is colored $i$ or $j$). The number of such cycles is denoted $d_{ij}(\Gamma,c)$. 

Let $C$ be a connected component of $\Gamma_{ij}(c)$. Swapping the colors $i,j$ of $c$ along $C$ provides a new coloring $c'$ of $\Gamma$. One says that $c$ and $c'$ are related by a \emph{Kempe move along $C$}. Two elements $c$ and $c'$ of $\Tait(\Gamma)$ are \emph{Kempe equivalent} if there exists a finite sequence of Kempe moves transforming $c$ into $c'$, possibly for different pairs $(i,j)$. Kempe equivalence is an equivalence relation, with equivalence classes called \emph{Kempe classes}. 

\begin{prop}[{\cite[Theorem 1]{FISK1977298}}]
\label{prop:bipartite-KE}
If a web $\Gamma$ is bipartite then all Tait colorings are Kempe equivalent. 
\end{prop}

A Kempe class $\kappa$ is called \emph{homogeneous} if  the sum 
\begin{equation}\label{eq:degree_class} 
d(\Gamma,c):=d_{12}(\Gamma,c) + d_{23}(\Gamma, c) + d_{13}(\Gamma,c)
\end{equation} 
is independent of the coloring $c \in \kappa$. This integer is called the \emph{degree of $\kappa$}, for a homogeneous $\kappa$, and  denoted by $d(\kappa)$ or $d(\Gamma,\kappa)$. 
\begin{itemize}
    \item 
A web is called \emph{weakly homogeneous} if it admits at least one homogeneous Kempe class.
\item 
A web is called \emph{semi-homogeneous} if every Kempe class of the web is homogeneous. 
\item 
A web $\Gamma$ is called \emph{homogeneous of degree $d$} if it is homogeneous and each Kempe class of $\Gamma$ has the same degree $d$. For such webs, $d(\Gamma):= d$ is called the \emph{Kempe-degree of $\Gamma$}. 
\end{itemize}

\begin{exa}\label{exa:graded-Kempe}
\begin{enumerate}
    \item The simplest homogeneous webs are the empty web, denoted $\emptyset_1$ and the circle web, of degrees $0$ and  $2$, respectively. The empty web has a unique Tait coloring and its Kempe class has degree $0$. The circle web admits three Tait colorings that are all Kempe equivalent. This unique Kempe class $\kappa$ is homogeneous  and $d(\kappa)=2$.
    \item The $\Theta$-web is the web with two vertices and three edges connecting  them. It admits six Tait colorings which are all Kempe equivalent. This unique Kempe class $\kappa$ is homogeneous, with $d(\kappa)=3$,  so the  $\Theta$-web is  homogeneous of Kempe-degree three.
    \item The tetrahedral web $K_4$ is the web obtained by embedding the complete graph on 4 vertices into the plane. It admits six Tait colorings which are all Kempe equivalent. The unique Kempe class $\kappa$ is homogeneous and $d(\kappa)=3$.
    \item The cubic web $C_3$ is the embedding of the 1-skeleton of the 3-dimensional cube in the plane. It admits 24 Tait colorings. Since it is bipartite, all its Tait colorings are Kempe equivalent. However, $C_3$ has Tait colorings $c$ and $c'$ with different degrees $d(C_3,c)=6$ and 
    $d(C_3,c')=4$, see formula (\ref{eq:degree_class}) and Figure~\ref{fig:cube-non-hom}.
    \begin{figure}
        \centering
        \NB{\tikz[scale=0.8]{\input{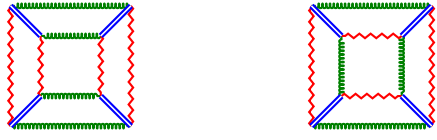}}}
        \caption{Two Tait colorings $c$ and $c'$ of the cube $C_3$. They are related by a Kempe move along the inner square. One has $d(C_3, c) = 2+ 2+2$ and $d(C_3, c') = 2+ 1 +1$.}
        \label{fig:cube-non-hom}
    \end{figure}
    This implies that the cube graph $C_3$ is not weakly homogeneous. 
    \item Section~\ref{sec:dodecahedron} investigates the dodecahedral graph $G_d$. The latter turns out to admit 60 Tait colorings which are partitioned into 10 Kempe classes, each homogeneous of degree 3. Consequently, $G_d$ is a homogeneous web of Kempe-degree three.  
\end{enumerate}
\end{exa}
\begin{figure}
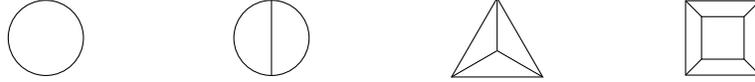

    \centering
    \NB{\tikz[]{\begin{scope}[scale =1]
  \begin{scope}[xshift = 0cm]
    \draw (0,0) circle (0.5cm);
  \end{scope}
  \begin{scope}[xshift = 3cm]
    \draw (0,0) circle (0.5cm);
    \draw (0, -.5) -- + (0,1);
  \end{scope}
  \begin{scope}[xshift = 6cm, yshift = -0.17cm]
    \draw (0,0) -- (90:0.7);
    \draw (0,0) -- (210:0.7);
    \draw (0,0) -- (330:0.7);
    \draw (90:0.7) -- (210:0.7) -- (330:0.7) -- (90:0.7);
  \end{scope}
  \begin{scope}[xshift = 9cm]
    \draw (45:0.4)   -- (45:0.7);
    \draw (-45:0.4)  -- (-45:0.7);
    \draw (135:0.4)  -- (135:0.7);
    \draw (-135:0.4) -- (-135:0.7);
    \draw (45:0.4) -- (135:0.4) -- (-135:0.4) -- (-45:0.4) -- (45:0.4);
    \draw (45:0.7) -- (135:0.7) -- (-135:0.7) -- (-45:0.7) -- (45:0.7);
  \end{scope}
\end{scope}}}
    \caption{From left to right: the circle, the $\Theta$-web, the tetrahedron and the cube.}
    \label{fig:exa-web}
\end{figure}

Define a \emph{cone} over a web $\Gamma$ in $\RR^2$ by considering $\RR^2$ as the boundary of $\RR^3_+$. Choose a point $p$ in the interior of $\RR^3_+$ and form the cone $\cone(\Gamma)$
over $\Gamma$ by connecting $p$ by straight intervals to all points of $\Gamma$. As a  topological space, $\cone(\Gamma)\cong \Gamma\times [0,1]\left/\Gamma\times\{0\}\right.$, and  comes with an embedding into $\RR^3$. The space $\cone(\Gamma)$ may have a singularity at $p$ which is more complicated than singular points of foams in~\cite{KR1}.  The cone inherits from $\Gamma$ the structure of a two-dimensional combinatorial CW-complex. The \emph{pointed cone}  $\cone_\star(\Gamma)$ is the pair $(\cone(\Gamma), p)$.

\subsection{Foams}
\label{sec:foams}
A \emph{conical foam} is a finite 2-dimensional CW-complex $F$ PL-embedded in $\RR^3$ such that for any point $p$ in $F$ there exists a small 3-ball $B$ centered in $p$ such that $(F\cap B, p)$ is PL-homeomorphic to a pointed cone over a connected web $\Gamma(p)$ in $\SS^2$. Note that the PL-isotopy type of $\Gamma(p)\subset \SS^2$ is well-defined. The web $\Gamma(p)$ is the link of the vertex $p$ of $F$. 


A point $p$ in $F$ is \emph{regular} if $\Gamma(p)$ is a circle, it is a \emph{seam point} if $\Gamma(p)$ is a $\Theta$-web, otherwise it is a \emph{singular point}. 

If for any singular point $p$ of $F$, $\Gamma(p)$ is a tetrahedral graph $K_4$, then this definition of a foam is the same as in \cite{KR1}. Such foams are called \emph{regular} foams. The union of seams and singular points of $F$ is denoted $s(F)$. The set $s(F)$ inherits from $F$ a structure of finite 1-dimensional CW-complex except possibly for several circles in $s(F)$ that do not have singular points on them. 
We view $s(F)$ as a graph (possibly with multiple edges, loops, and circular edges without vertices). Vertices of $s(F)$ are the  singular points of $F$. W call $s(F)$ the \emph{seam graph} of $F$. 

The connected components of $F\setminus s(F)$ are the \emph{facets} of $F$. Denote by $f(F)$ the set of facets of $F$. A foam may carry a finite number of dots on its facets, which can freely float on a facet but are not allowed to move across seams into adjacent facets.

An conical foam $F$ is \emph{adorned} if for every singular point $p$ of $F$ there is a preferred Kempe class $\kappa(p)$ of $\Gamma(p)$. An adorned conical foam $F$ is \emph{homogeneous} if for any singular point $p$ of $F$, $\kappa(p)$ is homogeneous.

\begin{figure}
    \centering
    \NB{\tikz[]{\input{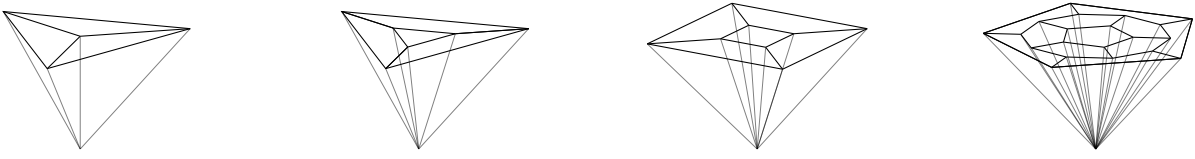}}}
    \caption{Example of possible singular points appearing in conical foams.}
    \label{fig:my_label}
\end{figure}
\subsection{Colorings}
\label{sec:colorings}

Let $F$ be an adorned conical foam. An \emph{admissible coloring} of $F$ (or simply a \emph{coloring}) is a map  $c\co f(F)\to \{1,2,3\}$ such that:
\begin{itemize}
    \item For each seam  of $F$, the three facets $f_1$, $f_2$ and $f_3$ adjacent to the seam carry different colors, that is, $\{c(f_1), c(f_2), c(f_3)\} = \{1,2,3\}$.
    \item For each singular point $p$ of $F$, the Tait-coloring $c_{\Gamma(p)}$ of $\Gamma(p)$ induced by $c$ is in the Kempe class $\kappa(p)$.
\end{itemize}
The set of admissible coloring of $F$ is denoted $\col{F}$.

Let $c$ be an admissible coloring of $F$. For  two distinct elements $i\neq j$ of $\{1,2,3\}$ denote by $\widetilde{F}_{ij}(c)$ the closure in $F$ of the union of facets $f$ such that $f(c) \in \{i,j\}$. This is a CW-complex embedded in $\RR^3$. Near seam points and regular points contained in $\widetilde{F}_{ij}(c)$, the space $\widetilde{F}_{ij}(c)$ is locally PL-homeomorphic to a disk. Near a singular point $p$, $\widetilde{F}_{ij}(c)$ is PL-homeomorphic to a cone over a disjoint union of $d_{ij}(\Gamma(p), c_{\Gamma(p)})$ circles.

To \emph{smooth out $\widetilde{F}_{ij}(c)$  at
$p$} means to replace this cone by a collection of $d_{ij}(\Gamma(p),
c_{\Gamma(p)})$ embedded disks by separating the discs that meet at the vertex $p$.  An example of this operation is given in  Figure~\ref{fig:smooth}.
Note that if $d_{ij}(\Gamma(p), c_{\Gamma(p)}) =1$, there is only one disk near $p$ and this procedure does not change the neighbourhood of $p$.  
\begin{figure}
    \centering
    \NB{\tikz[scale=0.75]{\input{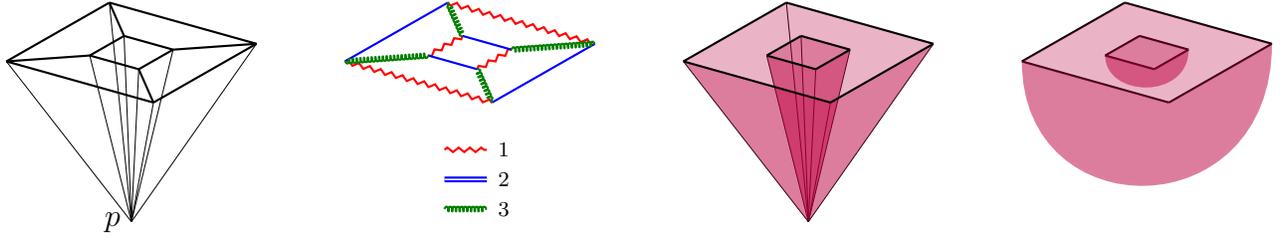}}}
    \caption{Smoothing out $\widetilde{F}_{12}(c)$ at a singular point $p$. From left to right: a neighborhood $U$ of singular point $p$; the coloring induced by $c$ on its link $\Gamma(p)$; the space $\widetilde{F}_{12}(c) \cap U$; the surface ${F}_{12}(c)\cap U$. Note, though, that graph $\Gamma(p)$ is not even weakly homogeneous.}
    \label{fig:smooth}
\end{figure}
The \emph{bicolored surface} $F_{ij}(c)$ is obtained by smoothing out $\widetilde{F}_{ij}(c)$ at every singular point of $F$. It is a surface embedded in $\RR^3$ and therefore has even Euler characteristic.

\begin{lem}
\label{degree}
For   an homogeneous adorned conical foam $F$ and its coloring $c$ the following relation holds:
\begin{align}\label{eq:degfoam}
\chi\left( F_{12}(c)\right) +
\chi\left( F_{13}(c)\right) +
\chi\left( F_{23}(c)\right) = 3\chi(s(F)) + 2\sum_{f\in f(F)} \chi(f) +  \sum_{p \in p(F)} (d(\kappa(p)) -3).
\end{align}
In particular, this quantity is independent from the admissible coloring $c$.
\end{lem}
\begin{proof}
Let us first prove that
\begin{align}\label{eq:degtilde}
\chi( \widetilde{F}_{12}(c)) +
\chi( \widetilde{F}_{13}(c)) +
\chi( \widetilde{F}_{23}(c)) = 3 \chi(s(F)) + 2\sum_{f\in f(F)} \chi(f) .
\end{align}

$\widetilde{F}_{12}(c)$ is a finite CW-complex which can be constructed starting with $s(F)$ and attaching to it $1$- and $2$-colored facets of $F$. Hence, 
\[
\chi(\widetilde{F}_{12}(c)) = \chi(s(F)) + \sum_{f \in f_1(F,c)} \chi(f) + \sum_{f \in f_2(F,c)}\chi(f),\]
where $f_i(F,c)$ denotes the set of $i$-colored facets of $F$ with the coloring $c$. The same holds for $\widetilde{F}_{13}(c)$ and $\widetilde{F}_{23}(c)$. Summing over these identities, one obtains (\ref{eq:degtilde}).

Let $p$ be a singular point of $F$. 
Smoothing out $\widetilde{F}_{ij}(c)$ at $p$ amounts to removing a point in $\widetilde{F}_{ij}$ and replacing it with $d_{ij}(\Gamma(p),c_{\Gamma(p)})$ points: it increases the Euler characteristic by $d_{ij}(\Gamma(p),c_{\Gamma(p)})-1$. This proves (\ref{eq:degfoam}). 
\end{proof}

The \emph{degree} $d(F)$ of an homogeneous adorned foam $F$ is given by the following formula:
\begin{equation}
d(F) = 2\#\{\text{dots on $F$}\} - 3\chi(s(F)) - 2\sum_{f\in f(F)} \chi(f) -  \sum_{p \in p(F)} (d(\kappa(p)) -3).
\end{equation} 


\subsection{Brief review of foam evaluation}
\label{sec:review-ev}
Foam evaluation was introduced in~\cite{RW1} and adapted by the authors~\cite{KR1} to the unoriented $SL(3)$ case. In this section we briefly review this construction in the generalized framework of homogeneous adorned conical foams. Let
\begin{equation}
    R'=\kk[X_1, X_2, X_3]
\end{equation}
be the graded ring of polynomials in $X_1,X_2,X_3$, with each $X_i$ in degree $2$. Denote by $E_1, E_2$ and $E_3$ the elementary symmetric  polynomials in $X_1,X_2,X_3$. They have degrees $2, 4$ and $6$ respectively and generate the subring of $S_3$-invariants 
\begin{equation}
     R \ := \ (R')^{S_3} \simeq \kk [E_1, E_2, E_3]. 
\end{equation}
Let  
\begin{equation}
    R'' \ := \  R'\left[ \frac{1}{X_i+X_j}\right]_{1\le i<j\le 3}.
\end{equation}

Finally, for $i\neq j$, denote by $\R''_{ij}\subset R''$ the graded ring $R'\left[ \frac{1}{(X_i + X_k)(X_j + X_k)}\right]$, for $\{i,j,k\}=\{1,2,3\}$.  One has $R'= R''_{12} \cap R''_{13} \cap R''_{23}$. There are inclusions of rings 
\begin{equation}
    R \subset R' \subset R'_{ij}\subset R''. 
\end{equation}




From now on a \emph{conical foam} means a  homogeneous adorned conical foam. Let $F$ be a conical foam and $c$ an admissible coloring of $F$.

Define 
\begin{align}
  Q(F,c) &= (X_1 + X_2)^{\frac{\chi(F_{12}(c))}{2}}  (X_2 +
  X_3)^{\frac{\chi(F_{23}(c))}{2}}  (X_1 +
  X_3)^{\frac{\chi(F_{13}(c))}{2}}, 
\label{eq:Q}
  \\ 
  P(F,c) &= \prod_{f\in f(F)} X^{n(f)}_{c(f)}, \label{eq:P}
\end{align}
where $n(f)$ denotes the number of dots on the facet $f$. The
\emph{evaluation} of the colored conical foam $(F,c)$ is the rational function $\kup{F,c}$ given by:
\[
  \kup{F,c}:= \frac{P(F,c)}{Q(F,c)} \in R''.
\]
The \emph{evaluation} of the conical foam $F$ is defined by:
\[
\kup{F} = \sum_{c\in \col{F}} \kup{F,c}.
\]

\begin{prop}[{\cite[Theorem 2.17]{KR1}}]
For any conical foam $F$ the evaluation $\kup{F}\in R$ is a symmetric polynomial in $X_1$, $X_2$ and $X_3$. It is homogeneous of degree $d(F)$.
\end{prop}

The proof in \cite{KR1} extends to conical foams in a straightforward way. The above evaluation coincides with the one in~\cite{KR1} for the usual foams. 

\subsection{State spaces}
\label{sec:staspa}

Using conical foam evaluation we construct three functors $\kup{\bullet}^c$, $\kup{\bullet}^c_E$ and $\kup{\bullet}_0^c$ from the category $\Foam^c$ described below to categories $R$-$\md$, $\kk[E]$-$\md$ and $\kk$-$\vect$. The ground field $\kk$ can be taken to be $\FF_2$, but occasionally it is useful to take an arbitrary  characteristic two field as $\kk$.

The category $\Foam^c$ has webs (either in $\RR^2$ or $\SS^2$) as objects. If $\Gamma_0$ and $\Gamma_1$ are two webs, then $\Hom_{\Foam^c}(\Gamma_0, \Gamma_1)$ consists of homogeneous adorned conical foams $U$ with boundary properly embedded in $\RR^2 \times [0,1]$ (respectively, in $\SS^2\times [0,1]$) such that \[\partial U = U\cap\{0,1\} = \Gamma_0 \times\{0\} \cup \Gamma_1 \times \{1\}.\]
Isotopic rel boundary conical foams define the same morphism in $\Foam^c$. 
Composition is given by concatenation. 

The category $\Foam$ is a subcategory of $\Foam^c$ that consists of foams in the usual sense, as in~\cite{KR1}. It has the same objects as $\Foam^c$, and a morphism $F$ in $\Foam^c$ is a morphism in $\Foam$ if for every singular point $p$ of $F$, $\Gamma(p)$ is the tetrahedral graph $K_4$. Note that since $K_4$ has only one Kempe class of Tait colorings, there is only one way to adorn singular point $p$. Such a singular point is the standard singular point on a foam as defined in~\cite{KR1}, and the  category $\Foam$ is the original category of unoriented $SL(3)$ foams in \cite{KR1}.

Functors $\kup{\bullet}^c$, $\kup{\bullet}^c_E$ and $\kup{\bullet}_0^c$ are defined via the  universal construction, see \cite{BHMV, SL3}, just as in \cite{KR1}. Let $S$ be a graded unital $\kk$-algebra and $\phi\co R \to S$ a homomorphism of graded unital commutative $\kk$-algebras. The homomorphisms considered here are
\begin{itemize}
\item $\mathrm{id}_R: R\lra R$, 
\item $\phi_E\co R \to \kk[E]$ given by $\phi_E(E_i) = \delta_{i3}E$, 
\item $\phi_0\co R \to \kk$ given by $\phi_0(E_i) =0$.
\end{itemize}

For a web $\Gamma$, define $W_S(\Gamma)$ to be the free graded $S$-module generated by $\Hom_{\Foam^c}(\emptyset_1, \Gamma)$.  Consider the symmetric bilinear form $(,)_\phi$ given on generators of $W_S(\Gamma)$ by:
\[
(F,G)_\phi :=  \phi(\kup{\overline{G}\circ F}) \in S, 
\]
where $\overline{G} \in \Hom(\Gamma, \emptyset_1)$ is the mirror image of $G$ with respect to the plane $\RR^2\times \{\frac12\}$ (respectively, $\SS^2\times \{\frac12\}$). Define 
\begin{equation}
    \kup{\Gamma}^c_\phi \ := \ W_S(\Gamma)/\mathrm{ker}((,)_{\phi})
\end{equation}
as the quotient of $W_S(\Gamma)$ by the kernel of $(,)_\phi$. These quotients extend to a functor
\begin{equation}
\kup{\bullet}^c_\phi\co \Foam^c \to S\textrm{-}\md.
\end{equation} 
Since the homomorphism assigned to a conical foam $F$ has degree $d(F)$, one can use the category of graded modules and homogeneous module homomorphisms, denoted $S\textrm{-}\md$ here, as the target category. This category is not pre-additive, since the sum of two homogeneous  homomorphisms of different degrees is not homogeneous. 

Functors $\kup{\bullet}^c$, $\kup{\bullet}^c_E$ and $\kup{\bullet}_0^c$ correspond to the choices $\phi = \mathrm{id}_R$, $\phi= \phi_E$ and $\phi=\phi_0$ respectively, see above.
Functors $\kup{\bullet}$, $\kup{\bullet}_E$ and $\kup{\bullet}_0$ and, more generally, $\kup{\bullet}_\phi$ are defined in \cite{KR1} in exactly the same manner but using the  category $\Foam$ instead of $\Foam^c$.

A closed conical foam $F$ is an endomorphism of the empty web $\emptyset_1$. $\kup{F}^c_\phi$ is then an endomorphism of $\kup{\emptyset_1}_\phi^c \cong S$, hence an element of $S$, and 
\[
\kup{F}^c_\phi = \phi\left(\kup{F}\right). 
\]


\begin{prop}\label{prop:subquotient-2}
For every web $\Gamma$ and every graded ring homomorphism $\phi\co R \to S$, the graded $S$-module $\kup{\Gamma}_\phi$ is a subquotient of $\kup{\Gamma}_\phi^c$.
\end{prop}

\begin{proof}
This follows from the general fact that if $W$ is an $S$-module endowed with a bilinear form and $V \subseteq W$ is a submodule, then $V/\ker (,)_{V}$ is a subquotient of $W/\ker (,)$. Indeed, there is the following diagram:
\[W\left/\ker (,) \right.\hookleftarrow  V\left/\left(\ker (, ) \cap V \right) \right. \twoheadrightarrow V \left/ \ker (, )_{|V}. \right.\qedhere
\]
\end{proof}

We can consider an intermediate theory  $\kup{\bullet}^i$ that mediates between  $\kup{\bullet}$ and $\kup{\bullet}^c$. 
Given a web $\Gamma$, define $\kup{\Gamma}^i$ as the subspace of 
$\kup{\Gamma}^c$ spanned by the usual foams in $\Foam$ with boundary $\Gamma$ and not by conical foams with more general singularities.  This collection of subspaces $\kup{\Gamma}^i$, over all $\Gamma$, extends to a functor from  $\Foam$ to the category of graded modules over $R$. The space $\kup{\Gamma}$ is naturally a quotient of $\kup{\Gamma}^i$. The inclusion and the quotient together constitute a diagram 
\begin{equation}\label{eq_two_maps} 
    \kup{\Gamma}^c \hookleftarrow \kup{\Gamma}^i \twoheadrightarrow \kup{\Gamma}
\end{equation}
that commutes with maps induced by  foams. Given a  foam $F$ with $\partial_0(F)=\Gamma_0$ and $\partial_1(F)=\Gamma_1$, there is a commutative diagram 
\begin{equation}
\NB{\begin{tikzpicture}
\node (G0c) at (0,0) {$\kup{\Gamma_0}^c$}; 
\node (G0i) at (2,0) {$\kup{\Gamma_0}^i$}; 
\node (G0)  at (4,0) {$\kup{\Gamma_0}$}; 
\node (G1c) at (0,-1.5) {$\kup{\Gamma_1}^c$}; 
\node (G1i) at (2,-1.5) {$\kup{\Gamma_1}^i$}; 
\node (G1)  at (4,-1.5) {$\kup{\Gamma_1}$}; 
\draw[->] (G0) -- (G1) node [midway, right] {$\kup{F}$};
\draw[->] (G0i) -- (G1i) node [midway, right] {$\kup{F}^c$};
\draw[->] (G0c) -- (G1c)node [midway, left] {$\kup{F}^c$};
\draw[left hook->] (G0i) -- (G0c);
\draw[->>] (G0i) -- (G0);
\draw[left hook->] (G1i) -- (G1c);
\draw[->>] (G1i) -- (G1);
\end{tikzpicture}}
\end{equation}
Consequently, there are natural transformations of functors 
\begin{equation}\label{eq_two_func} 
    \kup{\bullet}^c \hookleftarrow \kup{\bullet}^i \twoheadrightarrow \kup{\bullet},
\end{equation}
where these three flavours of $\kup{\bullet}$ are viewed as functors on the category $\Foam$. Note that $\kup{\bullet}^c$ is a functor on the larger category $\Foam^c$. 

We do not know a graph where even one of the maps in (\ref{eq_two_maps}) can be shown not to be an isomorphism. It is an open problem to compute the spaces and maps in  (\ref{eq_two_maps}) for the dodecahedral graph $\Gamma_d$. 

\vspace{0.1in} 
  
Many properties of functors $\kup{\bullet}_\phi$ proved in \cite{KR1} hold (with the same proofs) for functors $\kup{\bullet}^c_\phi$, including Propositions~3.10 to 3.16 in~\cite{KR1}.

\vspace{0.1in} 

Unoriented $SL(3)$ foams are remarkable in that we do not yet know a way to determine the state space of an arbitrary unoriented $SL(3)$ web. Oriented $SL(3)$ foams and webs are much better understood, and a basis in the state space of any oriented $SL(3)$ web can be constructed recursively, see~\cite{SL3,LHRThese,MackaayVaz,MoNi} and~\cite{Kuperberg}, where $SL(3)$ webs (Kuperberg webs) were introduced.


\section{State space of the dodecahedron}
\label{sec:dodecahedron}
In this section $S$ is a graded ring and $\phi\co R \to S$ is a graded ring homomorphism. If $M$ is a $\ZZ$-graded $S$-module, $M\{q^i\}$ denotes the graded $S$-module $M$ with the grading  shifted up by $i$. In other words, $(M\{q^i\})_j = M_{j-i}$. If $P = \sum_{i \in \ZZ} a_i q^i$ is a Laurent polynomial in $q$ with nonnegative integer coefficients, $M\{P\}$ denotes the $S$-module
\[
\bigoplus_{i \in \ZZ}  M\{q^i\}^{a_i}.
\]
The state space $\kup{\emptyset_1}^c_\phi$ of the empty web $\emptyset_1$ is a graded free $S$-module with the generator in degree zero represented by the  empty foam $\emptyset_2$, 
\begin{equation} \label{eq_empty_empty} 
\kup{\emptyset_1}^c_\phi\ \cong \ S \kup{\emptyset_2}_\phi \ \cong \ S.
\end{equation} 


\subsection{Theta-web}
Before inspecting state spaces of the dodecahedron, let us inspect state spaces of the $\Theta$-web, which
are well-understood without any  need for conical foams. 

\begin{prop}\label{prop:theta-space}
Let $\Gamma$ be the $\Theta$-web. Then the following isomorphism holds:
\[
\kup{\Gamma}^c_\phi \simeq \kup{\emptyset_1}^c_\phi\{(q^{-1}+q)(q^{-2} + 1 + q^{2})\}\simeq S\{[2][3]\}.
\]
\end{prop}
The above notations use quantum $[n]=q^{n-1}+q^{n-3}+\dots + q^{1-n}.$

This proposition follows directly from \cite[Propositions 3.12 and 3.13]{KR1}, but it is instructive to give a "one-step" proof. 

\begin{proof}[Proof of Proposition~\ref{prop:theta-space}]
Let us first show that foam evaluation satisfies the following local identity: 
\begin{align} \label{eq_theta_dec}
   \NB{\tikz[scale =0.4]{\begin{scope}
  \begin{scope}[yshift= 1.5cm]
    \draw (1,0) arc (0:360:1cm and 0.3cm)coordinate [pos=0](e) coordinate [pos=0.5](f) coordinate[pos=0.2] (a)
    coordinate[pos=0.7] (b);
    \draw (b) --(a);
  \end{scope}
  \begin{scope}[yshift =-1.5cm]
    \draw (1,0) arc (0:360:1cm and 0.3cm)coordinate [pos=0](g) coordinate [pos=0.5](h) coordinate[pos=0.2] (c)
    coordinate[pos=0.7] (d);
    \draw (d) --(c);
  \end{scope}
  \draw (a) -- (c);
  \draw (b) -- (d);
  \draw (e) -- (g);
  \draw (f) -- (h); 
\end{scope}}} =
\begin{array}{c} {\scriptstyle (2,1,0)}\\
   \NB{\tikz[scale = 0.4]{\begin{scope}
  \begin{scope}[yshift= 1.5cm]
    \draw (1,0) arc (0:360:1cm and 0.3cm) coordinate[pos=0.2] (a)
    coordinate[pos=0.7] (b);
    \draw[thin] (1,0) arc (0:-180:1cm and 1.2cm) coordinate[pos=0.5] (c);
    \draw (a) .. controls +(0,0) and +(0.2,0) .. (c) .. controls
    +(-0.2, 0) and +(0,0) .. (b) --(a);
  \end{scope}
  \begin{scope}[yshift =-1.5cm]
    \draw (1,0) arc (0:360:1cm and 0.3cm) coordinate[pos=0.2] (a)
    coordinate[pos=0.7] (b);
    \draw (1,0) arc (0:180:1cm and 1.2cm) coordinate[pos=0.5] (c);
    \draw (a) .. controls +(0,0) and +(0.2,0) .. (c) .. controls
    +(-0.2, 0) and +(0,0) .. (b) --(a);
  \end{scope}
\end{scope}}} \\ {\scriptstyle (0,0,0)}
\end{array}
+
\begin{array}{c} {\scriptstyle (2,0,0)}\\
   \NB{\tikz[scale = 0.4]{}} \\ {\scriptstyle (0,0,1)} 
\end{array}
+
\begin{array}{c} {\scriptstyle (1,1,0)}\\
   \NB{\tikz[scale = 0.4]{}} \\ {\scriptstyle (0,1,0) + (0,0,1)} 
\end{array}
+
\begin{array}{c} {\scriptstyle (1,0,0)}\\
   \NB{\tikz[scale = 0.4]{}} \\ {\scriptstyle (0,1,1) + (0,0,2)} 
\end{array}
+
\begin{array}{c} {\scriptstyle (0,1,0)}\\
   \NB{\tikz[scale = 0.4]{}} \\ {\scriptstyle (0,1,1)} 
 \end{array}
 +
\begin{array}{c} {\scriptstyle (0,0,0)}\\
   \NB{\tikz[scale = 0.4]{}} \\ {\scriptstyle (0,1,2)} 
\end{array}
\end{align}
where the triples of integers on the top and bottom  encode the number of dots on each of the three facets. Note that on the right-hand side of the identity, two middle terms are each a sum of two foams.

It is convenient to introduce some notations at this point: let us fix a web $\Gamma$, $v$ a vertex of $\Gamma$ and denote $e_1$, $e_2$ and $e_3$ the three edges of $\Gamma$ adjacent to $v$. For $i=1, \dots, 6$ define $x_i$ and $y'_i$ as the foams $\Gamma \times [0,1]$ with dots on the facets $e_1 \times [0,1]$, $e_2 \times [0,1]$, and $e_3 \times [0,1]$ given by Table~\ref{tab:xyp}.
\begin{table}[]
    \centering
\begin{tabular}{|c|c|c|c|} \hline
     Number of dots on: & $e_1 \times [0,1]$ &  $e_2 \times [0,1]$ & $e_3 \times [0,1]$ \\  \hline
  \hline  
  $x_1$ &2& 1& 0 \\ \hline
  $x_2$ &2& 0& 0 \\ \hline
  $x_3$ &1& 1& 0 \\ \hline
  $x_4$ &1& 0& 0 \\ \hline
  $x_5$ &0& 1& 0 \\ \hline
  $x_6$ &0& 0& 0 \\ \hline
  $y'_1$ &0& 0& 0 \\ \hline
  $y'_2$ &0& 0& 1 \\ \hline
  $y'_3$ &0& 1& 0 \\ \hline
  $y'_4$ &0& 0& 2 \\ \hline
  $y'_5$ &0& 1& 1 \\ \hline
  $y'_6$ &0& 1& 2 \\ \hline
\end{tabular}
    \caption{Definition of $x_i$ and $y'_i$.}
    \label{tab:xyp}
\end{table}

Finally, for $i = 1, \dots, 6$ define
\begin{align}
\label{eq:yyp}
\begin{cases}
 y_i := y'_i & \text{if $i \neq 3,4$,} \\
 y_3 := y'_2 + y'_3, &  \\
 y_4 := y'_3 + y'_4. &
\end{cases}
\end{align}
 Foams $x_i$ and $y_i$ are endomorphisms of $\Gamma$ in $\Foam$ and $\Foam^c$. In particular $x_6 = y_1 = \textrm{id}_\Gamma$. 

Choosing $v$ to be the top vertex of the $\Theta$-foam and the edges $e_1$, $e_2$ and $e_3$ going from left to right, we can rewrite (\ref{eq_theta_dec}):
\begin{align} \label{eq_theta_dec2}
   \NB{\tikz[scale =0.4]{\begin{scope}
  \begin{scope}[yshift= 1.5cm]
    \draw (1,0) arc (0:360:1cm and 0.3cm)coordinate [pos=0](e) coordinate [pos=0.5](f) coordinate[pos=0.2] (a)
    coordinate[pos=0.7] (b);
    \draw (b) --(a);
  \end{scope}
  \begin{scope}[yshift =-1.5cm]
    \draw (1,0) arc (0:360:1cm and 0.3cm)coordinate [pos=0](g) coordinate [pos=0.5](h) coordinate[pos=0.2] (c)
    coordinate[pos=0.7] (d);
    \draw (d) --(c);
  \end{scope}
  \draw (a) -- (c);
  \draw (b) -- (d);
  \draw (e) -- (g);
  \draw (f) -- (h); 
\end{scope}}} =\sum_{i=1}^6 \, 
\begin{array}{c} {x_i}\\
   \NB{\tikz[scale = 0.4]{}} \\ {y_i}
\end{array}
\end{align}

\emph{Local} identity  (\ref{eq_theta_dec}) (or \ref{eq_theta_dec2})) says that no matter how the foams in the equation are completed to closed foams by an outside conical foam $H$, the corresponding linear relation on the evaluations holds.  

Denote by $F$ the closed conical foam given by gluing $H$ onto the foam on the left-hand side of (\ref{eq_theta_dec2}). Denote by $G^i$ for $i \in \{1, \dots, 6\}$  six closed conical foams given by gluing $H$ to six foams on the right-hand side of (\ref{eq_theta_dec2}), going from left to right. Two of these six foams are actually sums of conical foams. 
Denote by $G$ the conical foam which is identical to $G_i$'s but without the dots that the latter carry in the region shown in (\ref{eq_theta_dec2}):
\[
G := \begin{array}{c} {\scriptstyle (0,0,0)}\\
   \NB{\tikz[scale = 0.4]{}} \\ {\scriptstyle (0,0,0)}
 \end{array}
\]
Foam $G$ can be viewed as the composition of two \emph{$\Theta$-half-foams}. These are two foams with $\Theta$-web as the boundary and a single singular seam. They are reflections of each other about a horizontal plane. Their composition in the opposite order is the usual $\Theta$-foam.

 Conical foams $G^1, \dots, G^6$ and $G$ are identical, except for the  distribution of dots. In particular, their sets of colorings are canonically isomorphic. The set $\col{G}$ can be partitioned into two subsets: colorings which restrict to the same colorings of the two $\Theta$-webs on the top and bottom of the identity and the ones which restrict to different colorings on the $\Theta$-webs on the top and bottom of the identity. The former is canonically isomorphic to $\col{F}$, the latter is denoted $\col{G}'$, so that one may write $\col{G} \cong \col{F} \sqcup \col{G}'$.

Let us prove that 
\begin{align}\label{eq:theta_eq_coloring}
\kup{F, c} &= \sum_{i=1}^6 \kup{G^i, c} \quad \text{for all $c \in \col{F}$}
\end{align}
and
\begin{align}
0 &= \sum_{i=1}^6 \kup{G^i, c} \quad \text{for all $c \in \col{G}'.$}\label{eq:theta_diff_coloring}
\end{align}

To show (\ref{eq:theta_eq_coloring}), fix $c\in \col{F}$ and view it as a coloring of $G$ and $G_i$'s as well.  Let the top and bottom $\Theta$-webs be colored by $c$ as follows, where $(i_1,i_2,i_3)$ is a permutation of $(1,2,3)$: 
\[
\cfthetacol[0.7]{i_1}{i_2}{i_3}.
\]
One has $\chi(F_{ij}(c)) = \chi(G_{ij}(c)) -2$ for all pairs $(i, j)$. Hence 
\[
\kup{F,c} = \kup{G,c}(X_{i_1} + X_{i_2})(X_{i_1} + X_{i_3})(X_{i_2} +  X_{i_3}).
\]
Furthermore,
\begin{align*}
    \sum_{i=1}^6 \kup{G_i, c}=& \kup{G,c}\left( X_{i_1}^2 X_{i_2} + X_{i_2}^2 X_{i_3} + X_{i_1}X_{i_2}(X_{i_2}+ X_{i_3}) + X_{i_1}(X_{i_2}X_{i_3} + X_{i_3}^2) \right. \\[-7pt] & \hspace{8.5cm} \left.+ X_{i_2}^2 X_{i_3} + X_{i_2}X_{i_3}^2\right) \\
    =& \kup{G,c} \left( X_{i_1}^2 X_{i_2} + X_{i_2}^2 X_{i_3} + X_{i_1}X_{i_2}^2+ X_{i_1}X_{i_3}^2 + X_{i_2}^2 X_{i_3} + X_{i_2}X_{i_3}^2\right) \\
    =& \kup{G,c} (X_{i_1}+X_{X_{i_2}})(X_{i_1}+ X_{i_3})(X_{i_2}+ X_{i_3}) \\ 
    =& \kup{F,c},
\end{align*}
which proves (\ref{eq:theta_eq_coloring}).

To show (\ref{eq:theta_diff_coloring}), fix $c \in \col{G}'$. Denote the induced Tait-colorings of the top and bottom  $\Theta$-web boundaries of $G$ as follows:
\[
\cfthetacol[0.7]{i_1}{i_2}{i_3} \qquad \text{and} \qquad
\cfthetacol[0.7]{j_1}{j_2}{j_3}.
\]

Then, 
\begin{align*}
    \sum_{i=1}^6 \kup{G_i, c}=& \kup{G,c}\left( X_{i_1}^2 X_{i_2} + X_{i_2}^2 X_{j_3} + X_{i_1}X_{i_2}(X_{j_2}+ X_{j_3}) + X_{i_1}(X_{j_2}X_{j_3} + X_{j_3}^2)\right. \\[-7pt] & \hspace{8cm}+ \left.X_{i_2}X_{j_2}X_{i_3} + X_{j_2}X_{j_3}^2\right) \\
    =& \kup{G,c} (X_{i_1}+{X_{j_2}})(X_{i_1}+ X_{j_3})(X_{i_{2}}+ X_{j_3}) \\ 
    =& 0.
\end{align*}
For the last equality, observe that $(i_1, i_2, i_3) \neq (j_1, j_2, j_3)$ and, consequently, one of the three linear in $X$ terms on the right hand side is zero. 

We next show that the terms on the right-hand side of (\ref{eq_theta_dec}) are pairwise orthogonal idempotents, as follows. Denote by $P_i$ for $i=1, \dots, 6$ the foams with boundary on the right-hand side of (\ref{eq_theta_dec}). Any two of them, say $P_i$ and $P_j$, compose by stacking one onto the other and gluing along the $\Theta$-web,  resulting in the foam $P_iP_j$. The foam evaluation satisfies the following local relation:
\begin{equation}\label{eq_orthogonality}
P_iP_j = \delta_{ij} P_i \, . 
\end{equation}
Note that $P_i P_j$ is always a sum of foams of the form:
\[
\begin{array}{c} {\scriptstyle (a_1,a_2,a_3)}\\
   \NB{\tikz[scale = 0.4]{}} \\ {\scriptstyle (a_4,a_5,a_6)}
 \end{array} \sqcup \Theta(k, \ell, m)\qquad
 \text{with} \qquad \Theta(k, \ell, m) :=
 \begin{array}{c} {\scriptstyle (k,\ell,m)}\\
   \NB{\tikz[scale = 0.4]{\begin{scope}[scale=1]
  \begin{scope}
    \draw[densely dotted] (1,0) arc (0:-180:1cm and 0.3cm) coordinate[pos=0.6] (b); 
    \draw[dotted] (1,0) arc (0:180:1cm and 0.3cm) coordinate[pos=0.4] (a); 
    \draw[thin] (1,0) arc (0:-360:1cm and 1.2cm) coordinate[pos=0.25] (c) coordinate[pos= 0.75] (d);
    \draw (a) .. controls +(0,-.3) and +(0.2,0) .. (c) .. controls
    +(-0.2, 0) and +(0,-.3) .. (b);
    \draw (a) .. controls +(0,.3) and +(0.2,0) .. (d) .. controls    +(-0.2, 0) and +(0,.3) .. (b);
    \end{scope}
\end{scope}}}\\ {\scriptstyle }
   \end{array},
\]
where $k, \ell$ and $m$ denote the number of dots on the left, right and middle facets of the $\Theta$-foam $\Theta(k, \ell, m)$.
Hence, (\ref{eq_orthogonality}) is a direct consequence of the multiplicativity of foam evaluation with respect to disjoint union and formulas for the evaluation of closed $\Theta$-foams. One can easily check each of the 36 cases using the following evaluation rules
\[
\kup{\Theta(k,\ell,m)}= 
\begin{cases}
1 & \text{if $\{k,\ell,m\} = \{0,1,2\}$,} \\
0 & \text{if $\#\{k,\ell,m\} \leq 2$},
\end{cases}
\]
$R$-linearity of closed foam evaluation, and various symmetries, to shorten the computation.

From orthogonality relations (\ref{eq_orthogonality}) and taking degrees into account, one deduces that the matrices
\begin{align*}
&\kup{\begin{pmatrix}
\begin{array}{c}
{\scriptstyle (2,1,0)}\\
   \NB{\tikz[scale = 0.4]{\begin{scope}
   \draw (1,0) arc (0:360:1cm and 0.3cm) coordinate[pos=0.2] (a)
    coordinate[pos=0.7] (b);
   \draw[thin] (1,0) arc (0:-180:1cm and 1.2cm) coordinate[pos=0.5] (c);
   \draw (a) .. controls +(0,0) and +(0.2,0) .. (c) .. controls
    +(-0.2, 0) and +(0,0) .. (b) --(a);
\end{scope}
}} 
\end{array}
&
\begin{array}{c} {\scriptstyle (2,0,0)}\\
   \NB{\tikz[scale = 0.4]{}}  
\end{array}
&
\begin{array}{c} {\scriptstyle (1,1,0)}\\
   \NB{\tikz[scale = 0.4]{}}  
\end{array}
&
\begin{array}{c} {\scriptstyle (1,0,0)}\\
   \NB{\tikz[scale = 0.4]{}} 
\end{array}
&
\begin{array}{c} {\scriptstyle (0,1,0)}\\
   \NB{\tikz[scale = 0.4]{}}  
 \end{array}
 &
\begin{array}{c} {\scriptstyle (0,0,0)}\\
   \NB{\tikz[scale = 0.4]{}}  
\end{array}
\end{pmatrix}}_\phi^c \qquad \text{and} \\
&
\kup{
\begin{pmatrix}
\begin{array}{c} \NB{\tikz[scale = 0.4]{  \begin{scope}
    \draw (1,0) arc (0:360:1cm and 0.3cm) coordinate[pos=0.2] (a)
    coordinate[pos=0.7] (b);
    \draw (1,0) arc (0:180:1cm and 1.2cm) coordinate[pos=0.5] (c);
    \draw (a) .. controls +(0,0) and +(0.2,0) .. (c) .. controls
    +(-0.2, 0) and +(0,0) .. (b) --(a);
  \end{scope}
}} \\ {\scriptstyle (0,0,0)}
\end{array}
&
\begin{array}{c} 
\NB{\tikz[scale = 0.4]{}} \\ {\scriptstyle (0,0,1)} 
\end{array}
&
\begin{array}{c}
   \NB{\tikz[scale = 0.4]{}} \\ {\scriptstyle (0,1,0) + (0,0,1)} 
\end{array}
&
\begin{array}{c}
   \NB{\tikz[scale = 0.4]{}} \\ {\scriptstyle (0,1,1) + (0,0,2)} 
\end{array}
&
\begin{array}{c}
   \NB{\tikz[scale = 0.4]{}} \\ {\scriptstyle (0,1,1)}
\end{array}
 &
\begin{array}{c}
   \NB{\tikz[scale = 0.4]{}} \\ {\scriptstyle (0,1,2)}
\end{array}
\end{pmatrix}^t}^c_\phi
\end{align*}
give mutually inverse isomorphisms between the state space $\kup{\Gamma}^c_\phi$ and 
\[\kup{\emptyset_1}_\phi^c\{(q^{-1} + q)(q^{-2} + 1 + q^2)\}\cong S\{ [2][3]\},
\]   
see also (\ref{eq_empty_empty}). 
In the definition of the matrices, the functor $\kup{F}_\phi^c$ is meant to be applied to each entry. Using $x_i$'s and $y_i$'s these matrices can be rewritten:
\begin{align}
\begin{pmatrix}
\kup{
\begin{array}{c} {x_i}\\
   \NB{\tikz[scale = 0.4]{}}  
\end{array}
}_{\phi}^c
\end{pmatrix}_{1\leq i \leq 6}
\quad \text{and} \quad
\begin{pmatrix}
\kup{
\begin{array}{c} 
   \NB{\tikz[scale = 0.4]{}}  \\
   { y_i}
\end{array}
}_{\phi}^c
\end{pmatrix}_{1\leq i \leq 6}
\end{align}
seen as line and column matrices respectively.
\end{proof}

From identity (\ref{eq_theta_dec}), one obtains that the state space of the $\Theta$-web $\Gamma$ is a free  $S$-module with a basis of foams
\[
\begin{array}{c}
{\scriptstyle (2,1,0)}\\
   \NB{\tikz[scale = 0.4]{}} 
\end{array}
,\quad
\begin{array}{c} {\scriptstyle (2,0,0)}\\
   \NB{\tikz[scale = 0.4]{}}  
\end{array}
,\quad
\begin{array}{c} {\scriptstyle (1,1,0)}\\
   \NB{\tikz[scale = 0.4]{}}  
\end{array}
,\quad
\begin{array}{c} {\scriptstyle (1,0,0)}\\
   \NB{\tikz[scale = 0.4]{}} 
\end{array}
,\quad
\begin{array}{c} {\scriptstyle (0,1,0)}\\
   \NB{\tikz[scale = 0.4]{}}  
 \end{array}
\quad \text{and} \quad
\begin{array}{c} {\scriptstyle (0,0,0)}\\
   \NB{\tikz[scale = 0.4]{}}  
\end{array}
.
\]
\vspace{0.1in}

Due to the $\Theta$-web $\Gamma$ admitting a symmetry axis which does not contain vertices of $\Gamma$, the state space $\kup{\Gamma}$ is a Frobenius algebra over $S$ rather than just an $S$-module. The multiplicative structure can be seen as coming from the foam depicted in Figure~\ref{fig:theta-pants}.

\begin{figure}[]ht
    \centering
    \NB{\tikz[]{\input{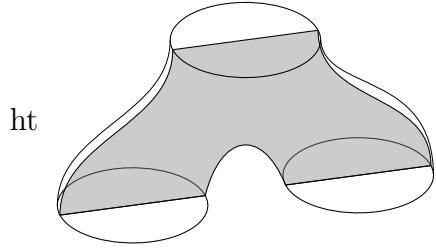}}}
    \caption{This pair of seamed pants gives the multiplicative structure on the state space of the $\Theta$-web.}
    \label{fig:theta-pants}
\end{figure}

The state space is a quotient of $S[Y_1, Y_2, Y_3]$, where 
\[
Y_1=
\begin{array}{c}
{\scriptstyle (1,0,0)}\\
   \NB{\tikz[scale = 0.4]{}} 
\end{array}
,\quad Y_2= 
\begin{array}{c}
{\scriptstyle (0,1,0)}\\
   \NB{\tikz[scale = 0.4]{}} 
\end{array}
\quad \text{and}\quad Y_3 =
\begin{array}{c}
{\scriptstyle (0,0,1)}\\
   \NB{\tikz[scale = 0.4]{}} 
\end{array}.
\]

Let us consider the case where $S=R= \kk[E_1, E_2, E_3]$ and $\phi=\mathrm{id}_R$. One can show that 
\[\kup{\Gamma}^c \simeq R[Y_1, Y_2, Y_3]\left/ \left(
\begin{array}{c} Y_1 + Y_2 + Y_3- E_1, \\ Y_1Y_2 + Y_2Y_3 + Y_1Y_3 - E_2, \\ Y_1Y_2Y_3 - E_3 \end{array}
\right) 
\right. \simeq \kk[Y_1, Y_2, Y_3] \simeq R'.\]
The state space $\kup{\Gamma}^c$ is a free $R$-module of rank six and can be identified with the $GL(3)$-equivariant cohomology of the full flag variety of $\CC^3$ with coefficients in $\kk$.  The ground ring $R$ is isomorphic to the $GL(3)$-equivariant cohomology of a  point with the same coefficient field.

\subsection{Tait colorings of the dodecahedron}

The dodecahedron graph $\Gd$ has 60 Tait colorings, and they correspond to 30 possible Hamiltonian cycles on $\Gd$. Each Hamiltonian cycle gives rise to six Tait colorings, and a coloring corresponds to a triple of Hamiltonian cycles, for colors $\{1,2\}$, $\{1,3\}$, $\{2,3\}$, correspondingly. Kempe moves permute these six colorings. For $\Gd$, Kempe equivalence classes coincide with the equivalence classes for the permutation action of $S_3$ on Tait colorings. 
For more details about such webs, see Section~\ref{sec:kempe}.

The group $A_5$ of rotations of the dodecahedron has order 60 and acts with two orbits on the set of Hamiltonian cycles. The two orbits consists of 15 'clockwise' and 15 'counterclockwise' Hamiltonian cycles. The group $A_5\times \ZZ_2$ of the isometries of the dodecahedron acts transitively on the set of Hamiltonian cycles.   
\begin{figure}[ht]
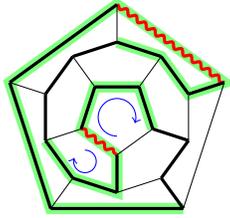

    \centering
    \NB{\tikz[]{\begin{scope}[scale={0.5},decoration={markings, mark=at
  position 0.5 with {\arrow{>}}},postaction={decorate}]
\begin{scope}[line width=1.5mm, green, opacity= 0.5]
\draw (198:2) -- (234:2)-- (270:2) --  (270:1) -- (198:1) -- (126:1) -- (54:1) -- (-18:1);
 \draw(90:2) -- (54:2) -- (18:2) -- (18:3) --( 90:3) --(162:3) -- (234:3) -- (306:3);
\end{scope}

\begin{scope}
\draw[->, blue] (270:0.5) arc (270:-18:0.5);
\draw[shift = (234:1.45cm), ->, blue] (90:0.3) arc (90: -198:0.3);
\end{scope}
  \begin{scope}[very thick]
    \draw[red, decoration={snake, segment length=1.5mm, amplitude=0.3mm},decorate] (18:3)  -- (90:3);
    \draw (162:3) -- ( 234:3);
    \draw (306:3) -- (306:2);
    \draw (54:1)  -- (126:1);
    \draw[red, decoration={snake, segment length=1.5mm, amplitude=0.3mm}, decorate] (198:1) -- (270:1);
    \draw (-18:1) -- (-18:2);
    \draw (18:2)  -- (54:2);
    \draw (90:2)  -- (126:2);
    \draw (162:2) -- (198:2);
    \draw (234:2) -- (270:2);
  \end{scope}
  \begin{scope}[very thick]
    \draw (306:2) -- (-18:2);
    \draw (198:2) -- (234:2);
    \draw (126:2) -- (162:2);
    \draw (54:2)  -- (90:2);
    \draw (270:1) -- (270:2);
    \draw (126:1) -- (198:1);
    \draw (-18:1) -- (54:1);
    \draw (18:3)  -- (18:2);
    \draw (234:3) -- (306:3);
    \draw (90:3)  -- (162:3);
  \end{scope}
  \begin{scope}[very thin]
    \draw (270:2) -- (306:2);
    \draw (-18:2) -- (18:2);
    \draw (198:1) -- ( 198:2);
    \draw (126:1) -- (126:2);
    \draw (54:1)  -- (54:2);
    \draw (270:1) -- (-18:1);
    \draw (234:3) -- (234:2);
    \draw (162:3) -- (162:2);
    \draw (90:3)  -- (90:2);
    \draw (306:3) -- (18:3);
  \end{scope}
\end{scope}}}
    \caption{A clockwise Hamiltonian cycle in $\Gd$.}
  \label{fig:hamitlonian-dodeca}
\end{figure}
See Figure~\ref{fig:hamitlonian-dodeca} above for one such cycle. It's determined by a pair of opposite edges in the dodecahedron graph (there are 15 such pairs, since there are 30 edges), and the direction of the 'twist' around the edge in a pair, either clockwise or counterclockwise, as the Hamiltonian cycle exits that edge. In Figure~\ref{fig:hamitlonian-dodeca} the cycle goes clockwise. 

The stabilizer subgroup in $A_5\times \ZZ_2$ for this cycle is $\ZZ_2\times \ZZ_2$, generated by the reflection in the plane through the two edges and by the reflection in the plane through the centers of the two edges and orthogonal to them. 

Group $A_5\times \ZZ_2$ has order $120$ and acts transitively on the set of Hamiltonian cycles, with stabilizers isomorphic to $\ZZ_2\times \ZZ_2$. Thus there are 30 Hamiltonian cycles, constituting one orbit of the $A_5\times \ZZ_2$ action.

\subsection{State space}
\label{sec:state-space-dodec-subsection}
\begin{prop}
  The following isomorphism of graded $R$-modules holds:
  \[ \kup{\Gd}^c \cong 10(q+ q^{-1})(q^2 + 1 + q^{-2})
  \kup{\emptyset_1}^c \cong R\{10[2][3]\}.
  \] 
  In particular $ \kup{\Gd}^c$ a  free $R$-module of graded rank $10[2][3]$.
\end{prop}

This decomposition holds for any homomorphism  $\phi: R\lra S$ of graded commutative rings as well, 
\[ \kup{\Gd}^c_{\phi}  \cong S\{10[2][3]\}.
  \] 

\begin{proof}

The proof uses a decomposition of the identity foam on $\Gd$ into a sum of 60 pairwise orthogonal idempotents. 

Recall that in the unoriented $SL(3)$ theory the state space  of the $\Theta$-web is a free rank six module over $R$, the state space of the empty graph. This decomposition can be written via a relation similar to the neck-cutting relation for 
the identity foam on the circle graph. The latter has three terms on the right hand side \cite{KR1} and mimics the formula 
\begin{equation}\label{eq_neck_cutting} 
 a = \sum_{i=1}^n x_i \otimes \epsilon_A(y_i a) 
\end{equation}
where $A$ is a commutative Frobenius algebra over a ground ring $R$ with dual bases $\{x_i\}_{i=1}^n$ and $\{y_i\}_{i=1}^n$ relative to the trace map $\epsilon_A : A \lra R$. 

Such a decomposition in the unoriented $SL(3)$ theory~\cite{KR1} for the state space $\kup{\Gamma}$ of a planar web $\Gamma$ always exists if $\kup{\Gamma}$ is a free graded module over $R$ (necessarily of finite rank), with the number of terms equal to the rank of the module.

Recall that $\Gd$ is homogeneous of Kempe-degree $3$ and has $10$
different Kempe equivalence classes. We denote these classes by $\kappa_j$, where $1\leq j \leq 10$, in some order.

Denote by $C_+ \subset \RR^3_+$ the conical foam which is the cone over $\Gd$. This
is a conical foam with a single singular point $p_+$. The bottom boundary of $C_+$ is $\Gd$ and the top boundary is empty. Denote
by $C_-$ the mirror image of $C_+$ with respect to reflection about
the horizontal plane and $p_-$ the singular point of $C_-$. 

Both $C_+$
and $C_-$ are unadorned conical foams with boundary with a single singular point each. We think
of $C_+$ as a singular cup and $C_-$ as a singular cap. Composing
these unadorned conical foams along the common boundary $\Gd$ results in a closed unardorned 
conical foam $C_+ C_-$ with two singular points (see Figure~\ref{fig:dodecahedron-capcup}). 

\begin{figure}[ht]
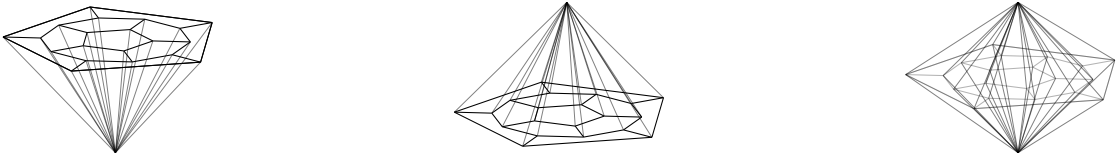

  \centering
  \NB{\tikz[]{\begin{scope}[yshift = 0cm, yscale = 0.15, xscale=0.5, rotate=13,decoration={markings, mark=at
  position 0.5 with {\arrow{>}}},postaction={decorate}]
\draw ( 18:3) -- ( 18:2);
\draw ( 90:3) -- ( 90:2);
\draw (162:3) -- (162:2);
\draw (234:3) -- (234:2);
\draw (306:3) -- (306:2);
\draw ( 18:3) -- ( 90:3);
\draw ( 90:3) -- (162:3);
\draw (162:3) -- (234:3);
\draw (234:3) -- (306:3);
\draw (306:3) -- ( 18:3);
\draw (-18:1) -- (-18:2);
\draw ( 54:1) -- ( 54:2);
\draw (126:1) -- (126:2);
\draw (198:1) -- (198:2);
\draw (270:1) -- (270:2);
\draw ( 18:3) -- ( 90:3) coordinate (a0);
\draw ( 90:3) -- (162:3) coordinate (a1);
\draw (162:3) -- (234:3) coordinate (a2);
\draw (234:3) -- (306:3) coordinate (a3);
\draw (306:3) -- ( 18:3) coordinate (a4);
\draw (-18:1) -- ( 54:1) coordinate (a5);
\draw ( 54:1) -- (126:1) coordinate (a6);
\draw (126:1) -- (198:1) coordinate (a7);
\draw (198:1) -- (270:1) coordinate (a8);
\draw (270:1) -- (-18:1) coordinate (a9);
\draw (-18:2) -- ( 18:2) coordinate (a10);
\draw ( 18:2) -- ( 54:2) coordinate (a11);
\draw ( 54:2) -- ( 90:2) coordinate (a12);
\draw ( 90:2) -- (126:2) coordinate (a13);
\draw (126:2) -- (162:2) coordinate (a14);
\draw (162:2) -- (198:2) coordinate (a15);
\draw (198:2) -- (234:2) coordinate (a16);
\draw (234:2) -- (270:2) coordinate (a17);
\draw (270:2) -- (306:2) coordinate (a18);
\draw (306:2) -- (-18:2) coordinate (a19);
\end{scope}
 \coordinate (t) at (0, -1.5);
 \coordinate (b) at (6, 0.5);
\begin{scope}[xshift=6cm, yshift =-1cm, yscale = 0.15, xscale=0.5, rotate=13,decoration={markings, mark=at
  position 0.5 with {\arrow{>}}},postaction={decorate}]
\draw ( 18:3) -- ( 18:2);
\draw ( 90:3) -- ( 90:2);
\draw (162:3) -- (162:2);
\draw (234:3) -- (234:2);
\draw (306:3) -- (306:2);
\draw (-18:1) -- (-18:2);
\draw ( 54:1) -- ( 54:2);
\draw (126:1) -- (126:2);
\draw (198:1) -- (198:2);
\draw (270:1) -- (270:2);
\draw ( 18:3) -- ( 90:3) coordinate (b0);
\draw ( 90:3) -- (162:3) coordinate (b1);
\draw (162:3) -- (234:3) coordinate (b2);
\draw (234:3) -- (306:3) coordinate (b3);
\draw (306:3) -- ( 18:3) coordinate (b4);
\draw (-18:1) -- ( 54:1) coordinate (b5);
\draw ( 54:1) -- (126:1) coordinate (b6);
\draw (126:1) -- (198:1) coordinate (b7);
\draw (198:1) -- (270:1) coordinate (b8);
\draw (270:1) -- (-18:1) coordinate (b9);
\draw (-18:2) -- ( 18:2) coordinate (b10);
\draw ( 18:2) -- ( 54:2) coordinate (b11);
\draw ( 54:2) -- ( 90:2) coordinate (b12);
\draw ( 90:2) -- (126:2) coordinate (b13);
\draw (126:2) -- (162:2) coordinate (b14);
\draw (162:2) -- (198:2) coordinate (b15);
\draw (198:2) -- (234:2) coordinate (b16);
\draw (234:2) -- (270:2) coordinate (b17);
\draw (270:2) -- (306:2) coordinate (b18);
\draw (306:2) -- (-18:2) coordinate (b19);
\end{scope}

\begin{scope}[very thin, opacity= 0.5]
  \foreach \i in {0,1,..., 19} {
  \draw (a\i) -- (t);
  \draw (b\i) -- (b);
  }
\end{scope}

\begin{scope}[xshift= 12cm, yshift = -0.5cm, yscale = 0.15, xscale=0.5, rotate=13,decoration={markings, mark=at
  position 0.5 with {\arrow{>}}},postaction={decorate},
  opacity = 0.3]
\draw ( 18:3) -- ( 18:2);
\draw ( 90:3) -- ( 90:2);
\draw (162:3) -- (162:2);
\draw (234:3) -- (234:2);
\draw (306:3) -- (306:2);
\draw ( 18:3) -- ( 90:3);
\draw ( 90:3) -- (162:3);
\draw (162:3) -- (234:3);
\draw (234:3) -- (306:3);
\draw (306:3) -- ( 18:3);
\draw (-18:1) -- (-18:2);
\draw ( 54:1) -- ( 54:2);
\draw (126:1) -- (126:2);
\draw (198:1) -- (198:2);
\draw (270:1) -- (270:2);
\draw ( 18:3) -- ( 90:3) coordinate (a0);
\draw ( 90:3) -- (162:3) coordinate (a1);
\draw (162:3) -- (234:3) coordinate (a2);
\draw (234:3) -- (306:3) coordinate (a3);
\draw (306:3) -- ( 18:3) coordinate (a4);
\draw (-18:1) -- ( 54:1) coordinate (a5);
\draw ( 54:1) -- (126:1) coordinate (a6);
\draw (126:1) -- (198:1) coordinate (a7);
\draw (198:1) -- (270:1) coordinate (a8);
\draw (270:1) -- (-18:1) coordinate (a9);
\draw (-18:2) -- ( 18:2) coordinate (a10);
\draw ( 18:2) -- ( 54:2) coordinate (a11);
\draw ( 54:2) -- ( 90:2) coordinate (a12);
\draw ( 90:2) -- (126:2) coordinate (a13);
\draw (126:2) -- (162:2) coordinate (a14);
\draw (162:2) -- (198:2) coordinate (a15);
\draw (198:2) -- (234:2) coordinate (a16);
\draw (234:2) -- (270:2) coordinate (a17);
\draw (270:2) -- (306:2) coordinate (a18);
\draw (306:2) -- (-18:2) coordinate (a19);
\end{scope}

 \coordinate (t) at (12, 0.5);
 \coordinate (b) at (12, -1.5);

\begin{scope}[very thin, opacity= 0.5]
  \foreach \i in {0,1,..., 19} {
  \draw (a\i) -- (t);
  \draw (a\i) -- (b);
  }
\end{scope}

}}
  \caption{From left to right, conical foams $C_+$, $C_-$ and $C_+C_-$.}
  \label{fig:dodecahedron-capcup}
\end{figure}
  
For each Kempe class $\kappa$ of $\Gd$, denote by $C_+(\kappa)$
(respectively $C_-(\kappa)$) the homogeneous adorned conical foam with boundary $C_+$ (respectively $C_-$),
where $p_+$ (respectively $p_-$) is mapped to $\kappa$.

The closed homogeneous adorned foam $C_+ (\kappa) C_- (\kappa')$ is given by composition of these two conical foams along the common boundary  $G_d$. It admits a Tait coloring if and only if $\kappa = \kappa'$. When this is the case, there are exactly six Tait colorings  of $C_+ (\kappa) C_- (\kappa)$. Furthermore, together they  constitute a single $S_3$ orbit on the set of Tait colorings of this homogeneous adorned conical foam.  

The unadorned conical foam with boundary $C_- C_+$ is the disjoint union of $C_-$ and $C_+$,  and it induces an endomorphism of $\Gd$. It has two singular points $p_-$ and $p_+$.   
For $j=1, \dots, 10$, denote by $P_j$  adorned conical foam with boundary $C_-(\kappa_j) C_+(\kappa_j)$. This is obtained by mapping both $p_-$ and $p_+$ to $\kappa_j$. Conical foam $P_j$ yields an endomorphism of the state space $\kup{\Gd}$.

Choose a vertex $v$ of $\Gd$ and denote its adjacent edges by 
$e_1, e_2$ and $e_3$. We use foams $x_i$'s and $y_i$'s ($1\leq i \leq 6$) introduced in the proof of Proposistion~\ref{prop:theta-space}. In our context they are endomorphisms of web $\Gd$ which differ from $\Gd \times [0,1]$ only by their dot distributions which are given by (\ref{eq:yyp}) and Table~\ref{tab:xyp}.


We claim that foam evaluation satisfies the following local relation:
\begin{align} \label{eq:dodec-id-split-1}
\Gd \times [0,1] &= \sum_{j=1}^{10} \sum_{i=1}^6 x_iP_j y_i, 
\end{align}
 and furthermore, the $6\times 10=60$ terms on the right hand side of the equation, when seen as endomorphisms of $\kup{\Gd}$, are pairwise
 orthogonal idempotents.

Diagramatically, this reads:
  \begin{align*}
   \NB{\tikz[scale =0.5]{\begin{scope}[yshift = 2cm, yscale = 0.15, xscale=0.5, rotate=13,decoration={markings, mark=at
  position 0.5 with {\arrow{>}}},postaction={decorate}]
\draw ( 18:3) -- ( 18:2);
\draw ( 90:3) -- ( 90:2);
\draw (162:3) -- (162:2);
\draw (234:3) -- (234:2);
\draw (306:3) -- (306:2);
\draw ( 18:3) -- ( 90:3);
\draw ( 90:3) -- (162:3);
\draw (162:3) -- (234:3);
\draw (234:3) -- (306:3);
\draw (306:3) -- ( 18:3);
\draw (-18:1) -- (-18:2);
\draw ( 54:1) -- ( 54:2);
\draw (126:1) -- (126:2);
\draw (198:1) -- (198:2);
\draw (270:1) -- (270:2);
\draw ( 18:3) -- ( 90:3) coordinate (a0);
\draw ( 90:3) -- (162:3) coordinate (a1);
\draw (162:3) -- (234:3) coordinate (a2);
\draw (234:3) -- (306:3) coordinate (a3);
\draw (306:3) -- ( 18:3) coordinate (a4);
\draw (-18:1) -- ( 54:1) coordinate (a5);
\draw ( 54:1) -- (126:1) coordinate (a6);
\draw (126:1) -- (198:1) coordinate (a7);
\draw (198:1) -- (270:1) coordinate (a8);
\draw (270:1) -- (-18:1) coordinate (a9);
\draw (-18:2) -- ( 18:2) coordinate (a10);
\draw ( 18:2) -- ( 54:2) coordinate (a11);
\draw ( 54:2) -- ( 90:2) coordinate (a12);
\draw ( 90:2) -- (126:2) coordinate (a13);
\draw (126:2) -- (162:2) coordinate (a14);
\draw (162:2) -- (198:2) coordinate (a15);
\draw (198:2) -- (234:2) coordinate (a16);
\draw (234:2) -- (270:2) coordinate (a17);
\draw (270:2) -- (306:2) coordinate (a18);
\draw (306:2) -- (-18:2) coordinate (a19);
\end{scope}

\begin{scope}[yshift = -2cm, yscale = 0.15, xscale=0.5, rotate=13,decoration={markings, mark=at
  position 0.5 with {\arrow{>}}},postaction={decorate}, very thin]
\draw ( 18:3) -- ( 18:2);
\draw ( 90:3) -- ( 90:2);
\draw (162:3) -- (162:2);
\draw (234:3) -- (234:2);
\draw (306:3) -- (306:2);
\draw (-18:1) -- (-18:2);
\draw ( 54:1) -- ( 54:2);
\draw (126:1) -- (126:2);
\draw (198:1) -- (198:2);
\draw (270:1) -- (270:2);
\draw ( 18:3) -- ( 90:3) coordinate (b0);
\draw ( 90:3) -- (162:3) coordinate (b1);
\draw (162:3) -- (234:3) coordinate (b2);
\draw (234:3) -- (306:3) coordinate (b3);
\draw (306:3) -- ( 18:3) coordinate (b4);
\draw (-18:1) -- ( 54:1) coordinate (b5);
\draw ( 54:1) -- (126:1) coordinate (b6);
\draw (126:1) -- (198:1) coordinate (b7);
\draw (198:1) -- (270:1) coordinate (b8);
\draw (270:1) -- (-18:1) coordinate (b9);
\draw (-18:2) -- ( 18:2) coordinate (b10);
\draw ( 18:2) -- ( 54:2) coordinate (b11);
\draw ( 54:2) -- ( 90:2) coordinate (b12);
\draw ( 90:2) -- (126:2) coordinate (b13);
\draw (126:2) -- (162:2) coordinate (b14);
\draw (162:2) -- (198:2) coordinate (b15);
\draw (198:2) -- (234:2) coordinate (b16);
\draw (234:2) -- (270:2) coordinate (b17);
\draw (270:2) -- (306:2) coordinate (b18);
\draw (306:2) -- (-18:2) coordinate (b19);
\end{scope}

\begin{scope}[very thin, opacity= 0.5]
  \foreach \i in {0,1,..., 19} { \draw (a\i) -- (b\i); }
\end{scope}
}} &= \sum_{j=1}^{10}
\begin{array}{c} {\scriptstyle (2,1,0)}\\
   \NB{\tikz[scale = 0.5]{\begin{scope}[yshift = 2cm, yscale = 0.15, xscale=0.5, rotate=13,decoration={markings, mark=at
  position 0.5 with {\arrow{>}}},postaction={decorate}]
\draw ( 18:3) -- ( 18:2);
\draw ( 90:3) -- ( 90:2);
\draw (162:3) -- (162:2);
\draw (234:3) -- (234:2);
\draw (306:3) -- (306:2);
\draw ( 18:3) -- ( 90:3);
\draw ( 90:3) -- (162:3);
\draw (162:3) -- (234:3);
\draw (234:3) -- (306:3);
\draw (306:3) -- ( 18:3);
\draw (-18:1) -- (-18:2);
\draw ( 54:1) -- ( 54:2);
\draw (126:1) -- (126:2);
\draw (198:1) -- (198:2);
\draw (270:1) -- (270:2);
\draw ( 18:3) -- ( 90:3) coordinate (a0);
\draw ( 90:3) -- (162:3) coordinate (a1);
\draw (162:3) -- (234:3) coordinate (a2);
\draw (234:3) -- (306:3) coordinate (a3);
\draw (306:3) -- ( 18:3) coordinate (a4);
\draw (-18:1) -- ( 54:1) coordinate (a5);
\draw ( 54:1) -- (126:1) coordinate (a6);
\draw (126:1) -- (198:1) coordinate (a7);
\draw (198:1) -- (270:1) coordinate (a8);
\draw (270:1) -- (-18:1) coordinate (a9);
\draw (-18:2) -- ( 18:2) coordinate (a10);
\draw ( 18:2) -- ( 54:2) coordinate (a11);
\draw ( 54:2) -- ( 90:2) coordinate (a12);
\draw ( 90:2) -- (126:2) coordinate (a13);
\draw (126:2) -- (162:2) coordinate (a14);
\draw (162:2) -- (198:2) coordinate (a15);
\draw (198:2) -- (234:2) coordinate (a16);
\draw (234:2) -- (270:2) coordinate (a17);
\draw (270:2) -- (306:2) coordinate (a18);
\draw (306:2) -- (-18:2) coordinate (a19);
\end{scope}
\coordinate (t) at (0, 0.5);
\node at (t) [left, font = \small] {$\kappa_j$};
\coordinate (b) at (0, -0.5);
\node at (b) [right, font = \small] {$\kappa_j$};
\begin{scope}[yshift = -2cm, yscale = 0.15, xscale=0.5, rotate=13,decoration={markings, mark=at
  position 0.5 with {\arrow{>}}},postaction={decorate}]
\draw ( 18:3) -- ( 18:2);
\draw ( 90:3) -- ( 90:2);
\draw (162:3) -- (162:2);
\draw (234:3) -- (234:2);
\draw (306:3) -- (306:2);
\draw (-18:1) -- (-18:2);
\draw ( 54:1) -- ( 54:2);
\draw (126:1) -- (126:2);
\draw (198:1) -- (198:2);
\draw (270:1) -- (270:2);
\draw ( 18:3) -- ( 90:3) coordinate (b0);
\draw ( 90:3) -- (162:3) coordinate (b1);
\draw (162:3) -- (234:3) coordinate (b2);
\draw (234:3) -- (306:3) coordinate (b3);
\draw (306:3) -- ( 18:3) coordinate (b4);
\draw (-18:1) -- ( 54:1) coordinate (b5);
\draw ( 54:1) -- (126:1) coordinate (b6);
\draw (126:1) -- (198:1) coordinate (b7);
\draw (198:1) -- (270:1) coordinate (b8);
\draw (270:1) -- (-18:1) coordinate (b9);
\draw (-18:2) -- ( 18:2) coordinate (b10);
\draw ( 18:2) -- ( 54:2) coordinate (b11);
\draw ( 54:2) -- ( 90:2) coordinate (b12);
\draw ( 90:2) -- (126:2) coordinate (b13);
\draw (126:2) -- (162:2) coordinate (b14);
\draw (162:2) -- (198:2) coordinate (b15);
\draw (198:2) -- (234:2) coordinate (b16);
\draw (234:2) -- (270:2) coordinate (b17);
\draw (270:2) -- (306:2) coordinate (b18);
\draw (306:2) -- (-18:2) coordinate (b19);
\end{scope}

\begin{scope}[very thin, opacity= 0.5]
  \foreach \i in {0,1,..., 19} {
  \draw (a\i) -- (t);
  \draw (b\i) -- (b);
  }
\end{scope}
}} \\ {\scriptstyle (0,0,0)} 
\end{array}
+
\begin{array}{c} {\scriptstyle (2,0,0)}\\
   \NB{\tikz[scale = 0.5]{}} \\ {\scriptstyle (0,0,1)} 
\end{array}
+
\begin{array}{c} {\scriptstyle (1,1,0)}\\
   \NB{\tikz[scale = 0.5]{}} \\ {\scriptstyle (0,1,0) + (0,0,1)} 
\end{array} 
+
\begin{array}{c} {\scriptstyle (1,0,0)}\\
   \NB{\tikz[scale = 0.5]{}} \\ {\scriptstyle (0,1,1) + (0,0,2)} 
\end{array}
+
\begin{array}{c} {\scriptstyle (0,1,0)}\\
   \NB{\tikz[scale = 0.5]{}} \\ {\scriptstyle (0,1,1)} 
 \end{array}
 +
\begin{array}{c} {\scriptstyle (0,0,0)}\\
   \NB{\tikz[scale = 0.5]{}} \\ {\scriptstyle (0,1,2)} 
\end{array}. 
\end{align*}
or
\begin{align}
    \NB{\tikz[scale =0.5]{}} &= \sum_{j=1}^{10} \sum_{i = 1}^6
\begin{array}{c} {x_i}\\
   \NB{\tikz[scale = 0.5]{}} \\ {y_i} 
\end{array}
\end{align}

Proof of this claim is identical to that of (\ref{eq_theta_dec})
except needing to take into account that $\Gd$ has ten distinct Kempe classes 
while the $\Theta$-foam has only one. This is why we need to sum over
these different Kempe classes.

Proving identity (\ref{eq:dodec-id-split-1}) is equivalent to proving that for
any conical foam $F$ bounding $\Gd\sqcup (-\Gd)$, the foam evaluation satisfies:
\begin{align}
     \kup{(\Gd\times[0,1])\cup F} &= \sum_{j=1}^{10} \sum_{i=1}^6
     \kup{x_iP_j y_i\cup F}
\end{align}


Fix a Kempe-class $\kappa_{j_0}$ with $1 \leq j_0 \leq 10$. Sets of admissible colorings of  conical foams $(x_i P_{j_0} y_i\cup F)_{1\leq i \leq 6}$  appearing in the sum are in natural bijections  with $\col{P_{j_0}\cup F}$. Let $c$ be an element of
$\col{P_{j_0}\cup F}$. It induces colorings of $C_+(\kappa_{j_0})$ and of $C_-(\kappa_{j_0})$
and therefore two colorings $c_+$ and $c_-$ of $\Gd$. Note that $c_+$ 
and $c_-$ are both in the Kempe-class $\kappa_{j_0}$ but might not be equal.

If $c_+ = c_-$, coloring $c$ induces an admissible coloring of
$(\Gd\times[0,1])\cup F$, still denoted $c$. The same computation as the one proving (\ref{eq:theta_eq_coloring}) gives:
\begin{align*}
     \kup{(\Gd\times[0,1])\cup F,c} &= \sum_{i=1}^6 \kup{x_i P_{j_0} y_i \cup F , c}.
\end{align*}

Note, furthermore, that any coloring of $(\Gd\times[0,1])\cup F$ is
obtained in this manner (for arbitrary $j_0$).

If, on the contrary, $c_+\neq c_-$, 
the same computation as the one proving (\ref{eq:theta_diff_coloring}) gives:
\begin{align*}
     \sum_{i=1}^6 \kup{x_i P_{j_0} y_i \cup F , c} =0.
\end{align*}

Summing over all Kempe classes and all colorings of $P_{j_0}\cup F$, we
obtain identity~(\ref{eq:dodec-id-split-1}).

It remains to show that the 60 terms on the right-hand side of  are pairwise orthogonal idempotents.   

Orthogonality of elements involving different Kempe classes follows from the absence of admissible colorings of the conical foam obtained by concatenating these elements. Idempotent property and orthogonality for elements involving the same Kempe class follow from the same computation proving identity~(\ref{eq_orthogonality}).
\end{proof}

There are two problems with extending our conical foam  evaluation to foams $F$ with more general singular vertices. First, the evaluation $\kup{F}$  may not be homogeneous. Second, the evaluation may have non-trivial denominators and be a rational function rather than a polynomial. 
For instance, consider
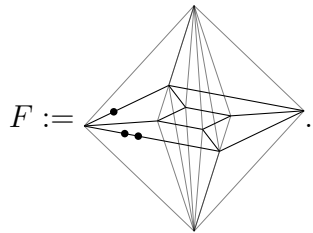
\[
F := \NB{\tikz[]{\begin{scope}
\begin{scope}[yshift = 2cm, yscale = 0.15, xscale=0.5, rotate = 13]
    \coordinate (A) at (  0: 3);
    \coordinate (B) at ( 90: 3);
    \coordinate (C) at (180: 3);
    \coordinate (D) at (270: 3);
    \coordinate (a) at (  0: 1);
    \coordinate (b) at ( 90: 1);
    \coordinate (c) at (180: 1);
    \coordinate (d) at (270: 1);
    \draw (A) -- (B);
    \draw (B) -- (C) coordinate[pos =0.65] (x);
    \draw (C) -- (D) coordinate[pos =0.3] (y) coordinate[pos =0.4] (z);
    \draw (D) -- (A);
    \draw (a) -- (b);
    \draw (b) -- (c);
    \draw (c) -- (d);
    \draw (d) -- (a);
    \draw (a) -- (A);
    \draw (b) -- (B);
    \draw (c) -- (C);
    \draw (d) -- (D);
  \end{scope}
  \coordinate (p) at (0, 0.5);
  \coordinate (q) at (0, 3.5);
    \begin{scope}[very thin, opacity = 0.5]
      \foreach \i in {A, B, C, D, a, b, c, d}{
      \draw (\i) -- (p);
      \draw (\i) -- (q);
      }
    \end{scope}
      \foreach \i in {x,y,z}{
      \fill (\i) circle (0.5mm);
      }
  \end{scope}
}}.
\]
Non-homogeneous conical foam $F$ consists of two cones over the square web glued along their common boundary with two dots on one facet, one dot on an adjacent facet, and no dots on the other facets. Since the square web has only one Kempe class of colorings, $F$ can be considered an adorned conical foam. A straightforward computation shows that foam evaluation $\kup{F}$ is given by the following formula:
\[
\kup{F} = \frac{1 + X_1^2 +X_2^2+ X_2^3 + X_3^2 + X_1X_2 + X_1X_3 + X_2X_3}{(X_1+X_2)(X_1+X_3)(X_2+ X_3)},
\]
which is neither a polynomial, nor a homogeneous rational function. 

Here is what goes wrong with evaluations for such more general conical foams. Suppose one is doing, say, a $(1,2)$-Kempe move on an $\SS^2$-component of $F_{12}(c)$ for some coloring $c$ of $F$ (these are the components that contribute to the denominators in $\kup{F,c}$). To ensure cancellation with $x_1+x_2$ in the denominator, one needs the balancing relation 
\begin{equation}\label{eq_balance}
    \chi(F_{13}(c))+\chi(F_{23}(c)) = \chi(F_{13}(c')) + \chi(F_{23}(c'))
\end{equation}
on the Euler characteristics of the bicolored surfaces for  $c$ and $c'$, see lemmas 2.12 and 2.18 in~\cite{KR1}. This relation, in general, fails to hold if we do not impose  homogeneity requirement (\ref{eq:degree_class}) from Section~\ref{sec:conical-foams} on the Kempe classes and restrict to such classes only. Quantity $d(\Gamma,c)$ in (\ref{eq:degree_class}) describes the contribution of the vertex which is a cone over $\Gamma$ to the sum of Euler characteristics 
\begin{equation*}
    \chi(F_{12}(c))+\chi(F_{13}(c))+\chi(F_{23}(c)) ,
\end{equation*}
see also Lemma~\ref{degree}, 
and this quantity needs to be invariant under Kempe moves to have a homogeneous and integral (polynomial) evaluation. For $c,c'$ related by a $(1,2)$-Kempe move, $F_{12}(c)\cong F_{12}(c')$, so the invariance reduces to that in (\ref{eq_balance}).


\section{Homogeneous webs}
\label{sec:kempe}
In Section~\ref{sec:tait-kempe} we introduced weakly homogeneous, semi-homogeneous and homogeneous webs. 
In this section we consider homogeneous webs and raise some questions
related to them. It is natural here to work with webs in $\SS^2$ rather than in  $\RR^2$.

We don't know much about homogeneous webs, in particular, we don't know if there exists a semi-homogeneous but not homogeneous web. 


\begin{dfn}
  For webs $\Gamma$ and $\Gamma'$  a map 
  $\varphi\co\col{\Gamma} \to \col{\Gamma'}$ \emph{intertwines Kempe
  equivalence} if two colorings  $c_1$ and $c_2$ of $\Gamma$ are Kempe equivalent if and only if $\varphi(c_1)$
  and $\varphi(c_2)$ are Kempe equivalent. 
\end{dfn}

In particular, for such a map, all the
  pre-images of a given coloring of $\Gamma'$ are Kempe equivalent.

The operation of converting a small neighbourhood of a vertex of a web
$\Gamma$ into a triangular region with three vertices and three
outgoing regions is called \emph{blowing up a vertex} of $\Gamma$.
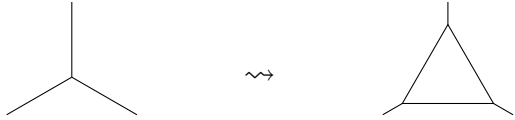
\[
  \NB{\tikz[]{\begin{scope}
    \begin{scope}
        \draw (0,0) -- (90:1);
        \draw (0,0) -- (-30:1);
        \draw (0,0) -- (-150:1);
    \end{scope}
    \node at (2.5, 0) {$\rightsquigarrow$}; 
    \begin{scope}[xshift = 5cm]
        \draw (90:0.7) -- (90:1);
        \draw (-30:0.7) -- (-30:1);
        \draw (-150:0.7) -- (-150:1);
        \draw (90: 0.7)--(-30: 0.7) -- (-150: 0.7)-- cycle;
    \end{scope}
\end{scope}}}
\]

\begin{prop}\label{prop_blow}
If $\Gamma'$ is obtained from $\Gamma$ by blowing up a
  vertex, then there exists a canonical bijection between the colorings of
  $\Gamma$ and $\Gamma'$. This bijection preserves the degree 
  and intertwines Kempe equivalence. In particular, $\Gamma$ is weakly
  homogeneous (resp.{} semi-homogeneous or homogeneous) if and
  only if $\Gamma'$ is.
\end{prop}

\begin{proof}
The bijection between colorings is given in Figure~\ref{fig:triangle-colorings}. Let $c$ be a coloring of $\Gamma$ and denote by $c$ the corresponding coloring of $\Gamma'$. For $1\le i < j\le 3$ there is a canonical bijection between connected components of $\Gamma_{ij}(c)$ and $\Gamma'_{ij}(c)$ (see Section~\ref{sec:colorings} for this notation).  Kempe moves along components of $\Gamma_{ij}(c)$ correspond to Kempe moves along components of $\Gamma'_{ij}(c)$.
\begin{figure}
    \centering
    \NB{\tikz[]{\input{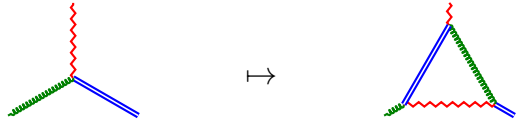}}}
    \caption{Bijection between colorings of $\Gamma$ and $\Gamma'$ in Proposition \ref{prop_blow}.}
    \label{fig:triangle-colorings}
\end{figure}

\end{proof}

\begin{prop}\label{prop:digon}
  Let $\Gamma$ and $\Gamma'$ be two webs which are identical except in
  a small ball where they are related as follows:
  \[
    \NB
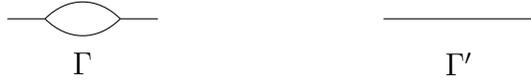
{\tikz[]{\begin{scope}
    \begin{scope}
        \draw (-1,0) -- (-0.5, 0) .. controls +(0.3, 0.3) and + (-0.3, 0.3) .. (0.5,0) -- (1,0);
        \draw (-0.5, 0) .. controls +( 0.3, -0.3) and + (-0.3, -0.3) .. (0.5,0);
        \node[below] at (0, -0.3) {$\Gamma$};
    \end{scope}
    \begin{scope}[xshift = 5cm]
        \draw (-1,0) -- (1,0);
        \node[below] at (0, -0.3) {$\Gamma'$};
    \end{scope}
\end{scope}}}.
  \]
  Then there exists a canonical 2-to-1 map from the set of colorings
  of $\Gamma$ to that of $\Gamma'$. This map  decreases the degree or a coloring by $1$ and intertwines Kempe
  equivalence. In particular, $\Gamma$ is weakly homogeneous (resp.{}
  semi-homogeneous or homogeneous) if and only if $\Gamma'$ is.
\end{prop}

\begin{proof}
The 2-to-1 map is given in Figure~\ref{fig:digon-colorings}. 
\begin{figure}
    \centering
    \NB{\tikz[]{\input{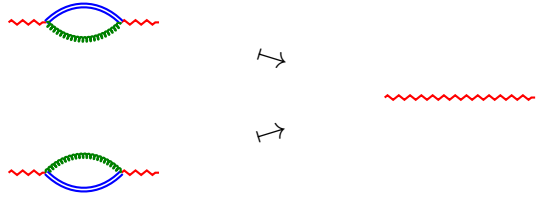}}}
    \caption{2-to-1 map between colorings of $\Gamma$ and of $\Gamma'$ of Proposition \ref{prop:digon}.}
    \label{fig:digon-colorings}
\end{figure}
Denote $i,j$ and $k$ the three elements of $\{1,2,3\}.$ Let $c$ is a coloring of $\Gamma'$ and denote by $c_1$ and $c_2$ the corresponding colorings of $\Gamma$. Let $c$ associate $i$ to the distinguished edge of $\Gamma'$. Colorings $c_1$ and $c_2$ are related by a Kempe move. Besides this Kempe move, all Kempe moves that one can do on $c_1$ corresponds canonically to Kempe moves on $c$ (and vice versa). Moreover, one has: $d_{ij}(\Gamma, c) = d_{ij}(\Gamma', c_1) = d_{ij}(\Gamma', c_2)$, $d_{ik}(\Gamma, c) = d_{ik}(\Gamma', c_1) = d_{ik}(\Gamma', c_2)$  and $d_{jk}(\Gamma, c) +1 = d_{jk}(\Gamma', c_1) = d_{jk}(\Gamma', c_2)$. 
\end{proof}

Among homogeneous webs, one class seems of peculiar interest: the webs for which Kempe classes coincide with orbits of the action of $S_3$ on the set of
colorings. 

\begin{dfn}\label{dfn:kempsmall}
A web $\Gamma$ is called \emph{Kempe-small} if $\col{\Gamma} \neq \emptyset$ and any Tait coloring $c$ of $\Gamma$ has the property that
its $(i,j)$-subcoloring $\Gamma_{ij}(c)$ is connected for all $1\le i < j \le 3$.
\end{dfn}

Examples of Kempe-small webs include the circle, the $\Theta$-web, the tetrahedral
web (both shown in Figure~\ref{fig:exa-web}) and the dodecahedral web, see
Figure~\ref{fig:dodecahedron}.
Note that a Kempe-small graph is necessarily connected. 

\begin{rmk}\label{rmk:kempe-small-abstract-graph}There is no obstruction to define \emph{Kempe-smallness}, \emph{weak homogeneity}, \emph{semi-homogeneity} or \emph{homogeneity} for abstract trivalent graphs. 
\end{rmk}


The next statement follows from Proposition~\ref{prop_blow}.

\begin{prop} If $\Gamma'$ is obtained from $\Gamma$ by blowing up a
  vertex, then one of $\Gamma$, $\Gamma'$ is Kempe-small if and only
  if the other one is.
\end{prop}


For $i=1,2$, consider a web $\Gamma_i$ on the 2-sphere $\SS^2$ and a vertex $v_i$ of $\Gamma_i$. A \emph{vertex-connected sum of $\Gamma_1$ and $\Gamma_2$ along $(v_1, v_2)$} is a web $\Gamma$ in $\SS^2$ obtained by removing small disk neighborhoods of $v_1,v_2$ and gluing the complements of this neighbourhoods together to get a trivalent graph embedded in $\SS^2$. There are six ways to form a vertex-connected sum along $v_1,v_2$  due to six ways to match three legs of $\Gamma_1$ at $v_1$ and three legs of $\Gamma_2$ at $v_2$. Three out of these six ways require reversing the orientation of one of $\SS^2$'s.

\begin{figure}[ht]
    \centering
    \NB{\tikz[]{\input{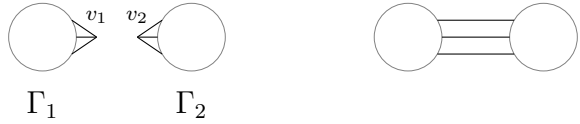}}}
    \caption{The disjoint union of $\Gamma_1$ and $\Gamma_2$ and one of their vertex-connected sum along $v_1$ and $v_2$.}    \label{fig:vertex-connected-sum2}
\end{figure}


The $\Theta$-web is a neutral element for the vertex-connected sum: vertex-connected sum of $\Gamma_1$ and $\Theta$ is isomorphic to  $\Gamma_1$. Performing a vertex-connected sum with the tetrahedral web amounts to blowing up a vertex.


\begin{prop}
  \label{prop:connected-sum}
  A vertex-connected sum of two webs $\Gamma'$ and $\Gamma''$ is
  Kempe-small if and only if both $\Gamma'$ and $\Gamma''$ are Kempe-small.
\end{prop}

\begin{proof}
Denote by $\Gamma$ the vertex-connected sum of $\Gamma'$ and $\Gamma''$. 
  Suppose that $\Gamma$ is Kempe-small. 
  Let $c'$ and $c''$ be
  colorings of $\Gamma'$ and $\Gamma''$ respectively. Applying an 
  $S_3$-symmetry to $c''$, if necessary, one can associate a coloring
  $c$ of $\Gamma$ to $c',c''$. Since $\Gamma$ is Kempe-small, $\Gamma_{ij}(c)$ is
  connected for any $\{i,j\}\subset \{1,2,3\}$. This implies that
  $\Gamma'_{ij}(c')$ and $\Gamma''_{ij}(c'')$ are connected for any
  $\{i,j\}\subset \{1,2,3\}$ and finally that $\Gamma'$ and $\Gamma''$
  are Kempe-small.

  Conversely, suppose that $\Gamma'$ and $\Gamma''$ are
  Kempe-small. Let $c$ be a coloring of $\Gamma$. We claim that it
  induces colorings $c'$ and $c''$ of $\Gamma'$ and $\Gamma''$. The
  only thing to show is that the three edges of $\Gamma$ involved in
  the vertex-connected sum are given different colors by $c$. 
  Choose
  any injective map
  $\varphi\co\{1,2,3\} \to \ZZ_2\times \ZZ_2 \setminus \{e\}$, where $\{e\}$ is the identity element of the abelian group $\ZZ_2\times \ZZ_2$. The map
  $\varphi \circ c \co E(\Gamma) \to \ZZ_2\times \ZZ_2 \setminus \{e\}$
  can be though of as an edge-coloring of $\Gamma$ by
  $\ZZ_2\times \ZZ_2 \setminus \{e\}$. It satisfies a flow condition:
  for any planar closed curve intersecting $\Gamma$ transversely, the
  product in $\ZZ_2\times \ZZ_2$ of the colors over all intersection points is $e$. 
  
  Indeed, this
  condition is trivially satisfied if the curve does not intersect
  $\Gamma$ and is invariant under pushing the curve through a vertex or an edge of
  $\Gamma$. The flow condition implies that the colors of the three
  edges involved in the connected sum are distinct. 

  Finally, $\Gamma_{ij}(c)$ is the connected sum of
  $\Gamma'_{ij}(c')$ and $\Gamma''_{ij}(c'')$ and is therefore
  connected.
\end{proof}

\begin{prop}
  \label{prop:KS-square}
  Let $\Gamma$, $\Gamma_v$ and $\Gamma_h$ be three webs which are
  identical except in a disk $B$ where they are related as follows:
  \[
    \NB
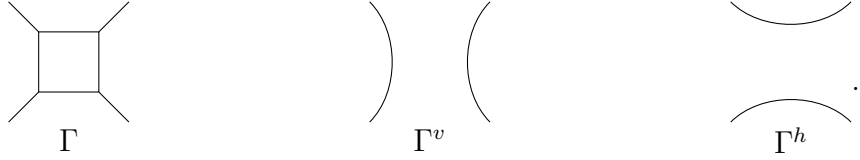
{\tikz[]{\begin{scope}[scale =0.8]
    \begin{scope}
        \draw (0.5,-0.5) -- (0.5,0.5) -- (-0.5, 0.5) -- (-0.5, -0.5) -- (0.5, -0.5);
        \draw (0.5, 0.5) -- (1,1);
        \draw (-0.5, 0.5) -- (-1,1);
        \draw (0.5, -0.5) -- (1,-1);
        \draw (-0.5, -0.5) -- (-1,-1);
        \node at (0, -1.3) {$\Gamma$};
    \end{scope}
    \begin{scope}[xshift = 6cm]
        \draw (1, -1) .. controls +(-0.5, 0.5) and +(-0.5, -0.5) .. (1,1);
        \draw (-1,  -1) .. controls +(0.5, 0.5) and +(0.5, -0.5) .. (-1,1);
        \node at (0, -1.3) {$\Gamma^v$};
    \end{scope}
    \begin{scope}[xshift = 12cm]
        \draw (1, -1) .. controls +(-0.5, 0.5) and +(0.5, 0.5) .. (-1,-1);
        \draw (1,  1) .. controls +(-0.5, -0.5) and +(0.5, -0.5) .. (-1,1);
        \node at (0, -1.3) {$\Gamma^h$};
    \end{scope}
\end{scope}}}.
  \]
  Then $\Gamma$ is Kempe-small if and only if one of the following holds:
\begin{itemize}
    \item both $\Gamma^v$ and $\Gamma^h$ are Kempe-small,
    \item $\Gamma^v$ is Kempe-small and $\Gamma^h$ admits no coloring,
    \item $\Gamma^h$ is Kempe-small and $\Gamma^v$ admits no coloring.
\end{itemize}
\end{prop}

\begin{proof}
  Suppose  first that $\Gamma$ is Kempe-small and let us prove that $\Gamma^v$ and $\Gamma^h$ fall into one of these three cases. Let $c$ be a coloring of $\Gamma$. Up to an $S_3$-symmetry, we can suppose that $c$ has one of the three types shown in Figure~\ref{fig:colorings-of-square}.
  \begin{figure}
      \centering
      \NB{\tikz[]{\input{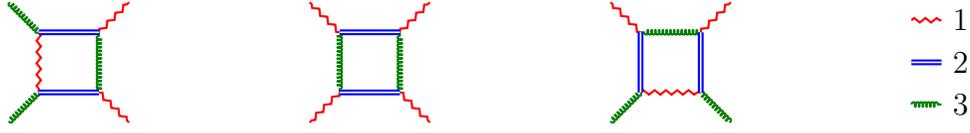}}}
      \caption{Possible colorings of the square.}
      \label{fig:colorings-of-square}
  \end{figure}
  This implies that either $\Gamma^v$, $\Gamma^h$ or both admit a Tait coloring. Hence it remains to show that if $\Gamma^h$ (resp.{} $\Gamma^v$) admits a coloring then it is Kempe-small. 
  
  Without loss of generality, one can assume that  $\Gamma^v$ admits a coloring $c_v$.
  Suppose that $c_v$ assigns the same color to
  the two edges of $\Gamma^v$ intersecting the disk $B$.  Up to an $S_3$-symmetry, we can
  suppose that these edges are colored by $1$. There are two Kempe-equivalent
  colorings $c_1$ and $c_2$ of $\Gamma$ which are induced by $c_v$ but 
  differ in $B$, see Figure~\ref{fig:cvc1c2}. 
  \begin{figure}[ht]
    \centering
    \NB{\tikz[]{\input{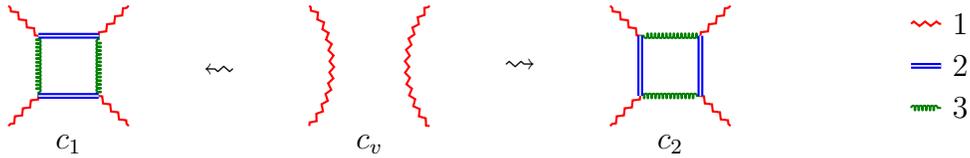}}}
    \caption{Coloring $c_v$ of $\Gamma^v$ induces two colorings $c_1$
    and $c_2$ of $\Gamma$.} \label{fig:cvc1c2}
  \end{figure}
  Since $\Gamma$ is Kempe-small, 
  $\Gamma_{12}(c_1)$ is connected. This in turns implies that
  $\Gamma_{12}(c_2)$ is not connected, see Figure~\ref{fig:c1c2}), and contradicts $\Gamma$ being Kempe-small.
  \begin{figure}[ht]
    \centering
    \NB{\tikz[]{\input{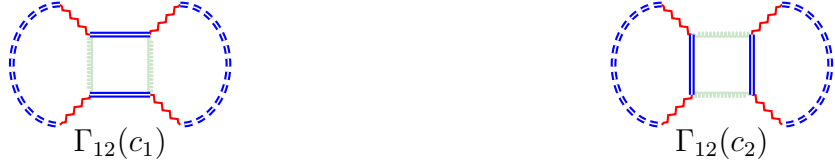}}}
    \caption{Colorings $c_1$ and $c_2$ of $\Gamma$: planarity of $\Gamma$ implies that one of $\Gamma_{12}(c_1)$ or $\Gamma_{12}(c_2)$ is not connected.} \label{fig:c1c2}
  \end{figure}
  Hence $c_v$  assigns different colors to the two edges in $B$. There is a unique coloring of $c$ of $\Gamma$ which coincides with
  $c_v$ except in $B$. Since $\Gamma$ is Kempe-small, $\Gamma_{12}(c)$, $\Gamma_{13}(c)$ and $\Gamma_{23}(c)$, are Hamitonian cycles, and $\Gamma^v_{12}(c_v)$, $\Gamma^v_{13}(c_v)$,  $\Gamma^v_{23}(c_v)$ are connected.

  Conversely, suppose that $\Gamma^v$  and $\Gamma^h$ satisfy one of the three conditions listed in the proposition.
  Let $c$ be a coloring of $\Gamma$. Suppose
  that it assigns the same color to the four edges pointing out of the square. Up to an  $S_3$-symmetry, one may suppose that this color is $1$. The coloring $c$ induces colorings $c_v$ and $c_h$ of
  $\Gamma^v$ and $\Gamma^h$ respectively, see Figure~\ref{fig:ccvch}. Hence we are in the case where $\Gamma^v$ and $\Gamma^h$ are Kempe-small.
  \begin{figure}[ht]
    \centering
    \NB{\tikz[]{\input{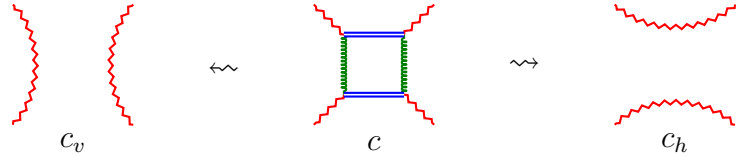}}}
    \caption{Coloring $c$ of $\Gamma$ induces coloring $c_v$ of
    $\Gamma_v$ and coloring $c_v$ of $\Gamma_v$.} \label{fig:ccvch}
  \end{figure}
  Graphs $\Gamma^h_{12}(c)$ and $\Gamma^v_{12}(c)$ are embedded in $\RR^2$ and related by a saddle move. Hence one of them is not connected. tThis is not possible with webs $\Gamma^v$ and $\Gamma^h$ being Kempe-small, thus not all the edges of the square have the same color in $c$.   
  Up to an $S_3$-symmetry, there are only two  possibilities for $c$ in $B$:
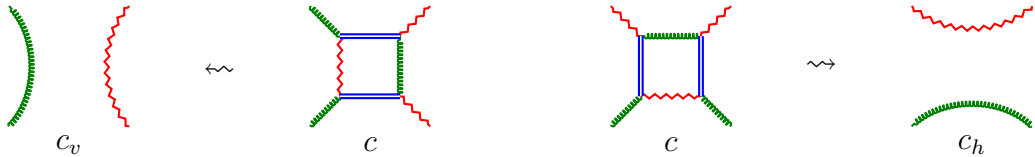
  \[
    \NB{\tikz[]{\begin{scope}[scale =0.8]
  \begin{scope}
    \node[black] at (0, -1.3) {$c$}; 
        \begin{scope}[green!50!black,decoration={coil,amplitude=0.3mm, segment length=0.5mm}, thick]
            \draw[decorate] (0.5,-0.5) -- (0.5,0.5);
            \draw[decorate] (-0.5, -0.5) -- (-1,-1);
            \draw[decorate] (-0.5, 0.5) -- (-1,1);
        \end{scope}
        \begin{scope}[red, decoration={zigzag, segment length=1.5mm, amplitude=0.3mm}, thick]
            \draw[decorate] (-0.5, 0.5) -- (-0.5, -0.5);
            \draw[decorate] (0.5, 0.5) -- (1,1);
            \draw[decorate] (0.5, -0.5) -- (1,-1);
        \end{scope}
        \begin{scope}[blue, thick]
            \draw[double] (0.5, 0.5) -- (-0.5, 0.5);
            \draw[double] (-0.5, -0.5) -- (0.5, -0.5);
        \end{scope}
    \end{scope} 
  \node at (-2.5, 0) {$\leftsquigarrow$};
  \node at (+7.5, 0) {$\rightsquigarrow$};    
  \begin{scope}[xshift = 5cm, rotate = 90]
        \node[black] at ( -1.3, 0) {$c$}; 
        \begin{scope}[green!50!black,decoration={coil,amplitude=0.3mm, segment length=0.5mm}, thick]
            \draw[decorate] (0.5,-0.5) -- (0.5,0.5);
            \draw[decorate] (-0.5, -0.5) -- (-1,-1);
            \draw[decorate] (-0.5, 0.5) -- (-1,1);
        \end{scope}
        \begin{scope}[red, decoration={zigzag, segment length=1.5mm, amplitude=0.3mm}, thick]
            \draw[decorate] (-0.5, 0.5) -- (-0.5, -0.5);
            \draw[decorate] (0.5, 0.5) -- (1,1);
            \draw[decorate] (0.5, -0.5) -- (1,-1);
        \end{scope}
        \begin{scope}[blue, thick]
            \draw[double] (0.5, 0.5) -- (-0.5, 0.5);
            \draw[double] (-0.5, -0.5) -- (0.5, -0.5);
        \end{scope}
    \end{scope} 
  \begin{scope}[xshift =10cm, rotate=90,  red, decoration={zigzag, segment length=1.5mm, amplitude=0.3mm}, thick]
    \node[black] at ( -1.3, 0) {$c_h$}; 
    \draw[red, decoration={zigzag, segment length=1.5mm,
    amplitude=0.3mm}, thick, decorate] (1, -1) .. controls +(-0.5, 0.5) and +(-0.5, -0.5) .. (1,1);
    \draw[green!50!black,decoration={coil,amplitude=0.3mm, segment length=0.5mm}, thick,decorate] (-1,  -1) .. controls +(0.5, 0.5) and +(0.5, -0.5) .. (-1,1);
  \end{scope}
  \begin{scope}[xshift = -5cm,  red, decoration={zigzag, segment length=1.5mm, amplitude=0.3mm}, thick]
    \node[black] at ( 0,-1.3) {$c_v$}; 
    \draw[red, decoration={zigzag, segment length=1.5mm,
    amplitude=0.3mm}, thick, decorate] (1, -1) .. controls +(-0.5, 0.5) and +(-0.5, -0.5) .. (1,1);
    \draw[green!50!black,decoration={coil,amplitude=0.3mm, segment length=0.5mm}, thick,decorate] (-1,  -1) .. controls +(0.5, 0.5) and +(0.5, -0.5) .. (-1,1);
  \end{scope}
\end{scope}}}
  \]
  In the first case, $c$ induces a coloring $c_v$ of $\Gamma^v$, hence $\Gamma^v$ is Kempe-small. Connectedness of $\Gamma_{12}(c)$, $\Gamma_{13}(c)$ and $\Gamma_{23}(c)$ follows from that of  $\Gamma^v_{12}(c_v)$, $\Gamma^v_{13}(c_v)$ and $\Gamma^v_{23}(c_v)$. 
  
  In the second case, $c$ induces a coloring $c_h$ of $\Gamma^h$, hence $\Gamma^h$ is Kempe-small. Connectedness of $\Gamma_{12}(c)$, $\Gamma_{13}(c)$ and $\Gamma_{23}(c)$ follows from that of  $\Gamma^h_{12}(c_h)$, $\Gamma^h_{13}(c_h)$ and $\Gamma^h_{23}(c_h)$. 
   
\end{proof}


Beyond the smallest examples of a circle and the theta-web, 
Kempe-small webs can be built inductively from the tetrahedral and  dodecahedral webs by blowing up vertices and forming vertex-connected sums. We don't know other examples of Kempe-small webs.


\bibliographystyle{abbrvurl}
\bibliography{biblio}

\end{document}